\documentclass[review,onefignum,onetabnum]{siamart190516}

\usepackage{mathtools} 
\usepackage{lipsum}
\usepackage{amsfonts}
\usepackage{graphicx}
\usepackage{epstopdf}
\usepackage{algorithmic}
\usepackage{subcaption}

\ifpdf
  \DeclareGraphicsExtensions{.eps,.pdf,.png,.jpg}
\else
  \DeclareGraphicsExtensions{.eps}
\fi

\newsiamremark{remark}{Remark}
\newsiamremark{hypothesis}{Hypothesis}
\crefname{hypothesis}{Hypothesis}{Hypotheses}
\newsiamthm{claim}{Claim}

\headers{Learning optimal multigrid smoothers via neural networks}{Ru Huang, Ruipeng Li, and Yuanzhe Xi}

\title{Learning optimal multigrid smoothers via neural networks\thanks{Submitted to the editors \today.
\funding{This work was supported by NSF grant OAC 2003720.}}}

\author{Ru Huang\thanks{Department o Mathematics, Emory University, Atlanta, GA}
\and Ruipeng Li\thanks{Center  for Applied  Scientific Computing,
    Lawrence  Livermore National  Laboratory,  P. O.  Box 808,  L-561,
    Livermore,   CA  94551   {(\tt{li50@llnl.gov})}.  This   work  was
    performed under the  auspices of the U.S. Department  of Energy by
    Lawrence    Livermore   National    Laboratory   under    Contract
    DE-AC52-07NA27344.}
\and Yuanzhe Xi\footnotemark[2]}

\usepackage{amsopn}

\def\RR{\mathbb{R}}

\def\inv{^{-1}}%
\def\trans{^{\mathsf{T}}}

\def\Ran{\textnormal{\mbox{Ran}}}

\providecommand{\norm}[1]{\lVert#1\rVert}

\newcommand{\eq}[1]{\begin{equation}\label{#1}}
\newcommand{\en}{\end{equation}}
\begin{document}
\maketitle
\begin{abstract}
Multigrid methods are one of the most efficient techniques for solving large sparse linear systems arising from Partial Differential Equations (PDEs) 
and graph Laplacians from machine learning applications. One of the key components of multigrid is smoothing, 
which aims at reducing high-frequency errors on each grid level. 
However, finding optimal smoothing algorithms is problem-dependent and can impose challenges for many problems. 
In this paper, we propose an efficient adaptive framework for learning optimized smoothers from operator stencils
in the form of convolutional neural networks (CNNs).
The CNNs are trained on small-scale problems from a given type of PDEs based on a supervised loss function derived from multigrid convergence theories,
and can be applied to large-scale problems of the same class of PDEs.
Numerical results on anisotropic rotated Laplacian problems and variable coefficient diffusion problems
demonstrate improved convergence rates and solution time compared with classical hand-crafted relaxation methods.
\end{abstract}

\section{Introduction}
%Partial Differential Equations (PDEs) play important roles in modeling various physical, biological and engineering phenomena. 
Partial Differential Equations (PDEs) play important roles in modeling various phenomena in many fields of science and engineering. 
%When PDEs do not have closed-form solutions, their solutions can be approximated by numerical methods.
Their solutions are typically computed numerically, when the closed-form solutions are not easily available, which
 leads to large-scale and ill-conditioned sparse linear systems to solve.  
In machine learning applications such as spectral clustering, graph-based semi-supervised learning and transportation network flows, 
solving large-scale linear systems associated with graph Laplacians is often needed. 
The development of efficient  linear solvers  is still an active research area nowadays \cite{Saad2020IterativeMF,DBLP:journals/actanum/Wathen15,smash}. 

Among many numerical solution schemes, multigrid methods often show superior efficiency and scalability especially for solving elliptic-type PDE and graph Laplacian problems \cite{brandt1984algebraic,ruge1983algebraic,briggs2000multigrid,trottenberg2000multigrid}. 
Fast convergence of multigrid  is achieved by exploiting  hierarchical grid structures to eliminate errors of all modes by smoothing and coarse-grid correction 
at each grid level. Thus, the performance of multigrid methods  highly depends on the smoothing property of a chosen smoother. 
However, the design of optimal smoothing algorithm is problem-dependent and often too complex to be achieved even by domain experts. 
In this paper, we propose an adaptive framework for training optimized smoothers via convolutional neural networks (CNNs), 
which directly learns a mapping from operator stencils to the inverses of the smoothers. 
The training process is guided by multigrid convergence theories for good smoothing properties on eliminating high-frequency errors.
Multigrid solvers equipped with the proposed smoothers inherit the convergence guarantees 
and scalability from standard multigrid algorithms and can show improved performance on anisotropic 
rotated Laplacian problems that are typically challenging for classical multigrid methods.
Numerical results demonstrate that a well-trained CNN-based smoother   
can damp high-frequency errors more rapidly and thus lead to a faster convergence of multigrid than traditional relaxation-based smoothers. 
Another appealing property of the proposed smoother and the training framework is 
the ability of generalization to problems of much larger sizes and more complex geometries.

\subsection{Related work}
There is an increasing interest in leveraging machine learning techniques to solve PDEs in the past few years. Several researchers have proposed to use
machine learning techniques to directly approximate the solutions of PDEs. For example, \cite{lagaris1998artificial} first proposed to use neural networks (NNs) to approximate
the solutions for both Ordinary Differential Equations (ODEs) and PDEs with a fixed boundary condition. Later,
\cite{tang2017study} utilized CNNs to solve Poisson equations with a simple geometry and \cite{berg2018unified} extended the techniques to more complex geometries.
\cite{han2018solving,sirignano2018dgm}  applied machine learning techniques to solve high dimensional PDEs, and \cite{wei2019general} focused on applying reinforcement
learning to solve nonlinear PDEs. \cite{sun2003solving} used parameterized realistic volume conduction models to solve Poisson equations and \cite{holl2020learning} trained a
NN to plan optimal trajectories and control the PDE dynamics and showed numerical results for solving incompressible Navier-Stokes equations.

Orthogonal to the above methods, a few studies have focused on leveraging NNs to improve the performance of existing solvers. For example, \cite{schmitt2019optimizing} developed optimization techniques for
geometric multigrid based on evolutionary computation.
\cite{mishra2018machine} generalized existing numerical methods as NNs with a set of trainable parameters. \cite{katrutsa2017deep} proposed a deep learning method to optimize
the parameters of prolongation and restriction matrices in a two-grid geometric multigrid scheme by minimizing the spectral radius of the iteration matrix.
\cite{greenfeld2019learning} used NNs to learn prolongation matrices in multigrid  in order to solve diffusion equations without retraining and \cite{luz2020learning}
generalized this framework to algebraic multigrid (AMG) for solving unstructured problems.

Meanwhile, researchers have also explored relationships between CNNs and differential equations to design better NN architectures.
For instance, \cite{he2019mgnet} designed MgNet which uses multigrid techniques to improve CNNs.
\cite{haber2018learning,chang2017multi} scaled up CNNs by interpreting the forward propagation as nonlinear PDEs.

Here, we would like to highlight the work \cite{hsieh2018learning}, which proposes to use CNNs and U-net \cite{RFB15a}
to learn a correction term to Jacobi method for solving Poisson equations. This approach is shown to preserve strong correctness and convergence guarantees. Since multigrid methods are known to be more scalable than Jacobi, we extend this idea to improve multigrid methods by designing optimal smoothers in this paper.
%We follow the multigrid convergence theories and propose an efficient training procedure.
%The trained smoothers show strong smoothing effects for high-frequency errors and lead to improved performance on challenging anisotropic rotated Laplacian problems.
%In addition, the gained performance is insensitive to the choice of coarsening and restriction schemes used.
To the best of our knowledge, our approach is the first attempt to use CNNs to learn the smoother at each level of multigrid with more than two levels
and exhibits good generalization properties to  problems with different sizes, geometries and variable coefficients.

The outline of the paper is organized as follows. In Section \ref{sec:background}, we review the background of the multigrid method and its convergence results. In Section \ref{Learning}, we propose an adaptive learning framework for learning optimized smoothers for constant coefficient PDEs on structured meshes and extend this framework to variable coefficient problems in Section \ref{sec:VariableCoeff}. We provide interpretation of the learned smoothers in Section \ref{inter} and demonstrate the performance of the proposed methods through extensive numerical examples in Section \ref{sec:exp}. Finally, we draw some conclusions in Section \ref{sec:conclusion}.

\section{Preliminaries and theoretical background}
\label{sec:background}
In this section, we review the classical convergence theory of iterative methods for solving the following linear system of equations 
\begin{equation} \label{eq:Au=f}
 Au =f,
\end{equation}
where $A\in \RR^{n\times n}$ is symmetric positive definite (SPD) and 
$u,f \in \RR^{n}$.  
Iterative methods generate  a sequence of improving approximations  to the solution of \eqref{eq:Au=f},  in which the approximate solution $u_k$ at iteration $k$ depends on the previous ones. Formally, an iterative solver can be expressed as:
\begin{equation}
u_{k}=\Phi(u_0,f,k),
\end{equation}
where the solver $\Phi: \mathbb{R}^{n}\times\mathbb{R}^{n}\times\mathbb{Z}\rightarrow \mathbb{R}^{n}$ is an operator that takes the initial guess $u_0$, right-hand side vector $f$ and generates $u_k$ at iteration $k$.

\subsection{Relaxation methods}
%(and forcing term as input and output the next approximation \red{RL:?}).
Iterations based on relaxation schemes can be written as
\begin{align}\label{eq:iteration}
u_{k+1} &= (I-M\inv A) u_{k} + M\inv f \notag \\
        &= G u_{k} + M\inv f, \quad G=I-M\inv A,
\end{align}
where $M$ is the relaxation matrix and $G$ is the iteration matrix.
Standard relaxation approaches include weighted Jacobi method with $M=\omega\inv D$ where $D$ denotes the diagonal of $A$ and
Gauss-Seidel method with $M=D-L$ where $-L$ is the strict lower triangular part of $A$.
Denoting by $e_k=u_{\ast}-u_k$ the error at iteration $k$, where $u_{\ast}$ is the exact solution of \eqref{eq:Au=f}, it follows that
%\begin{equation}
$e_k = G^k e_0$.
%\end{equation}
The following theorem gives a general convergence result for $\lim_{k\rightarrow \infty}{e_k}=0$.

\begin{theorem}[{\cite[Theorem 4.1]{saadbook2}}]
Denote by $\rho(G)$ the spectral radius of $G$.
The iteration \eqref{eq:iteration} converges for any initial vector $u_0$ if and only if $\rho(G)<1$.
\label{thm:fixed}
\end{theorem}
Notice that $\rho(G)$ represents the asymptotic convergence rate, 
which, however, does not, in general, predict error reduction for a few iterations \cite{briggs2000multigrid}. When relaxation methods are used as multigrid smoothers, they are typically applied $O(1)$ times in each smoothing step. Thus, the convergent smoothers defined as follows can guarantee a better smoothing effect.
%Instead, 
%whereas matrix norm $\norm{G}$ can be a better worst-case estimate (\cite{briggs2000multigrid}).
\begin{definition}[Convergent smoother in energy norm]
Assuming $A$ is SPD, relaxation matrix $M$ is called a \emph{convergent smoother} in the energy norm if $\norm{Ge_k}_A < \norm{e_k}_A, \forall e_k$, where $G=I-M^{-1}A$ and $\norm{x}_A^2={x\trans Ax}$.
\end{definition}
It can be shown that $M$ is a convergent smoother if and only if $\norm{G}_A<1$ or $M\trans+M-A$ is SPD. Since $\rho(G)$ is easier to compute than $\norm{G}_A$ and $\rho(G)<1$ is a necessary condition for both asymptotic convergence and single-iteration convergence, $\rho(G)$ is still often used as a metric of convergence rate of smoothers. 

Though relaxation schemes can have very slow convergence when being used as a solver, they are known to be very efficient
for smoothing the error. That is, after a few iterations, the remaining error varies slowly relative to the mesh grid, and thus can be approximated well on a coarser grid. This property is explored in multigrid methods as discussed in the next section.

\subsection{Multigrid methods}
%Convergence theory of multigrid methods has been well studied in literature, see, e.g., (\cite{BRANDT198623,amge,gamg,xu_zikatanov_2017}).  In this section, we overview the two-grid multigrid theory based on \cite{gamg}.
Multigrid methods exploit a hierarchy of grids with exponentially decreasing numbers of degrees of freedom on coarser levels, starting with the original problem on the finest level. On each level, the computational cost is proportional to the problem size, therefore, the overall complexity is still linear.
Smoothing and coarse-grid correction are the two main components of multigrid, which are designed to be complementary to each other in order to achieve fast convergence, i.e., they aim at eliminating ``high-frequency'' (oscillatory) and ``low-frequency'' (smooth) errors respectively,  where high- and low-frequency errors usually correspond to eigenvectors of $M\inv A$ with large and small eigenvalues.
Relaxation-based approaches such as weighted Jacobi and Gauss-Seidel are typical choices of multigrid smoothers  as these methods 
are inexpensive to apply and can effectively remove high-frequency errors for elliptic type PDEs.
On the other hand, the effectiveness of coarse-grid correction on low-frequency errors is due to the fact that  smooth errors can be interpolated accurately. 

When dealing with hard problems such as ones with irregular anisotropy, anisotropy not aligned along the coordinate axes, or complex geometries,
 efficiency of traditional smoothers can deteriorate, in which cases, stronger and often more expensive smoothers  are  needed such as
block smoothers \cite{doi:10.1080/00207169008803864,https://doi.org/10.1002/pamm.201210311}, ILU-based smoothers \cite{doi:10.1137/0910043} and smoothers based on Krylov methods \cite{doi:10.1137/0722038,Lin2020}.
Nevertheless, finding robust and efficient smoothers still remains a challenging problem for multigrid.

%For instance, interpolation
%operators in standard AMG can interpolate constant vectors exactly, 
%and methods like smooth aggregation AMG construct interpolation matrices from the smoothed errors directly.

Convergence theory of two-grid methods has been well studied \cite{BRANDT198623,amge,gamg,xu_zikatanov_2017} through the error propagation operator $E_{\mathrm{TG}}$ of the form:
%Error propagation operators  with post-smoothing in two-grid methods can be written as
\begin{equation} \label{eq:twogrid}
E_{\mathrm{TG}} = (I-M\inv A)(I-P(P\trans A P)\inv P\trans A),
\end{equation}
where $M$ is the smoother, $P\in \RR^{n\times n_c}$ is the prolongation operator, $P\trans$ is typically used as the restriction operator for symmetric problems, and $P\trans A P$ is the 
Galerkin coarse-grid operator. In general, smaller $\Vert E_{\mathrm{TG}}\Vert_A$ indicates faster convergence for two-grid methods.%on $\Ran(P)^{\perp_A}$ and has the kernel space $\Ran(P)$. 

In this paper we choose standard prolongation operators $P$ and only focus on using CNNs to parameterize $M$. 
The following theorem summarizes the main convergence result in \cite{gamg} with respect to $M$ and $P$.
%guarantees the existence of convergent smoothers in this case. [RL:??}
\begin{theorem}[{\cite{gamg}}]
Assuming $M\trans + M -A$ is SPD, denote by
\begin{equation} \label{eq:symM}
\tilde M = M\trans (M\trans + M -A)\inv M,
\end{equation}
the symmetrized smoother. Let $R\in \RR^{n_c \times n}$ be any matrix such that $RP=I$ and
\begin{equation} \label{eq:K}
  K = \max_{e\neq 0} \frac{\norm{(I-PR)e}_{\tilde M}^2}{\norm{e}_A^2}.
\end{equation}
We have $K \ge 1$ and $\norm{E_{\mathrm{TG}}}_A \le \left(1 - {1}/{K}\right)^{1/2}.$
\end{theorem} 
The quantity $K$ in \eqref{eq:K}, which is the so-called \emph{weak approximation property} \cite{citeulike:13797291}, essentially measures how accurately interpolation approximates the eigenvectors of $M\inv A$ proportional to the corresponding eigenvalues.
%RL: I don't want to call it ``asymptotic convergence rate''. This bound K can be loose, even in the ideal case
The optimal $K$ yields an \emph{ideal uniform bound of convergence rate},
which is often used to analyze convergence rate of smoothers in two-grid methods \cite{ultra_smooth}. 
%that can be computed explicitly. 
\begin{definition}[Ideal uniform convergence bound]
\label{eq:Kstar2}
Suppose $P$ takes form  $P=\begin{psmallmatrix} W \\ I \end{psmallmatrix}$ as in standard multigrid algorithms, 
where $R=\begin{pmatrix} 0 & I \end{pmatrix}$ and $S\trans=\begin{pmatrix} I & 0 \end{pmatrix}$. 
Denoting by $K_*$ the minimum $K$  in \eqref{eq:K} over $P$, 
we define quantity $\beta_{*}$ such that
\begin{equation} \label{eq:beta_star}
\beta_{*}^2=(1-1/K_{*}) = [1-{\lambda_{\min}((S\trans \tilde M S)\inv (S\trans A S))}],
\end{equation}
which can be considered as the ideal uniform bound of convergence rate \cite{gamg}.
%This quantity is often used to analyze convergence rate of smoothers in two-grid methods \cite{ultra_smooth}. 
\end{definition}

Extension from two-grid methods to multigrid methods is straightforward. This can be done by recursively applying two-grid methods on the coarse-grid system, see Algorithm~\ref{alg:multigrid} for a brief description of standard multigrid V-cycle. Notice that the smoother $M^{(l)}$ at level $l$ is only required to eliminate errors that are $A^{(l)}$-orthogonal to $\Ran(P^{(l)})$ in order to have fast convergence. This property will be used to design efficient training strategies for learning neural smoothers in the next section.

\begin{algorithm}[h]
   \caption{Multigrid V-cycle for solving $Au=f$}
   \label{alg:multigrid}
\begin{algorithmic}[1]
%    \STATE {\bfseries Input:} Linear system $Au=f$ with initial guess $u_0$. Multigrid hierarchy: for each level $l$, coefficient matrix $A^{(l)}$, approximation $u^{(l)}$,  right-hand side $f^{(l)}$, smoother $M^{(l)}$ and prolongation matrix $P^{(l)}$, where $A^{(0)}=A$, $f^{(0)}=f$ and $u^{(0)}=u_0$
%    \STATE {\bf Output:} $u^{(0)}$
%   \REPEAT
   \STATE Pre-smoothing: $u^{(l)}= u^{(l)} + (M^{(l)})^{-1} (f^{(l)} - A^{(l)}u^{(l)})$
   \STATE Compute fine-level residual: $r^{(l)}=f^{(l)}-A^{(l)}u^{(l)}$, and restrict it to the coarse level: $r^{(l+1)}=(P^{(l)})\trans r^{(l)}$  
   \IF{$l+1$ is the last level}
   %\STATE The coarse level problem is $A^{(t-1)}=R^{(t)}A^{(t)}P^{(t)}$
   \STATE Solve $A^{(l+1)}u^{(l+1)}=r^{(l+1)}$
   \ELSE 
   \STATE Call multigrid V-cycle recursively with $l=l+1$, $f^{(l+1)}=r^{(l+1)}$ and $u^{(l+1)}=0$
   \ENDIF
   \STATE Prolongate the coarse-level approximation and correct the fine-level approximation: $u^{(l)}=u^{(l)} + P^{(l)}u^{(l+1)}$
   \STATE Post-smoothing: $u^{(l)}= u^{(l)} + (M^{(l)})^{-1} (f^{(l)} - A^{(l)}u^{(l)})$
%   \UNTIL{$||r||<\epsilon$}
\end{algorithmic}
\end{algorithm}

\section{Learning deep neural smoothers for constant coefficient PDEs}\label{Learning}
The convergence of multigrid V-cycle heavily depends on the choice of smoothers.  Classical off-the-shelf smoothers such as weighted Jacobi or Gauss-Seidel exhibit near-optimal performance on simple Poisson equations and generally lose their efficiency on other types of PDEs.  In this section, we formulate the design of smoothers as a learning task and train a single neural network to
parameterize the action of the inverse of the smoother at a given grid level for constant coefficient PDEs discretized on structured meshes. 
The learned smoothers are represented as a sequence of convolutional layers and trained in an adaptive way guided by the multigrid convergence theory. 

\subsection{Formulation}
We define a PDE problem as the combination of PDE class $\mathcal{A}$,  forcing term $\mathcal{F}$ and boundary condition $\mathcal{G}$. To solve the problem numerically on a 2-D square domain, we discretize it on a grid of size $N\times N$, which leads to solving linear system $Au=f$ where $A\in \mathbb{R}^{N^{2}\times N^{2}}$ and $f\in \mathbb{R}^{N^{2}}$.
%Then the components of the discretized problem become ($A,f$) where the discretized matrix $A\in \mathbb{R}^{N^{2}\times N^{2}}$ and the right-hand 
%side vector $f\in \mathbb{R}^{N^{2}}$.  The solution of the discretized problem $u^{\ast}$ satisfies $Au^{\ast}=f$. 
Our goal is to train smoothers $M^{(0)},\dots,M^{(L-1)}$ on the first $L$
levels of a multigrid solver that has $L+1$ levels.
We assume here that the multigrid solver uses the same smoother for both the pre-smoothing and post-smoothing steps (c.f., lines 1 and 9 in \cref{alg:multigrid}, respectively), and uses direct methods as the coarsest-level solver. 
Denoting by $\Phi^{(0)}$ the multigrid hierarchy from level $0$, the training objective for $\Phi^{(0)}$ is to minimize the error
%$$\lim_{k\rightarrow\infty}\Phi_{M^{(1)},\dots,M^{(t)}}
%(u_{0},f,k)=u^{\ast},$$
\begin{equation} \label{eq:overalltrain}
\norm{\Phi^{(0)}(u_{0},f,k)-u_{\ast}}_2
%=\norm{\Phi_{M^{(0)},\dots,M^{(L-1)}}(u_{0},f,k)-u^{\ast}}_2,
\end{equation}
where $u_{0}$ is a given initial guess, $u_\ast$ is the exact solution, and $u_k=\Phi^{(0)}(u_{0},f,k)$ is the approximate solution 
by performing $k$ steps of V-cycles with $\Phi^{(0)}$.  

The advantage of minimizing \cref{eq:overalltrain} instead of the norm of the associated iteration matrix is that \cref{eq:overalltrain} can be evaluated and optimized more efficiently. For example, in two-grid methods, $\Phi^{(0)}(u_{0},f,k)-u_{\ast}=E_{TG}^k e_0$ for each exact solution $u_{\ast}$ and an arbitrary initial guess $u_0$. When multiple initial guesses are used to minimize \cref{eq:overalltrain} jointly with different iteration number $k$, the convergence property of the trained smoother can be justified by the following theorem,
which shows that when the loss of \cref{eq:overalltrain} is small, the norm of the associated two-grid operator, $E_{\mathrm{TG}}$, should also be small. 
It is easy to see that this property also holds true for multigrid  operators.

\begin{theorem}[{\cite{doi:10.1137/S0895479893243876}}]\label{thm:sbound}
For any matrix $X\in \mathbb{R}^{n\times n}$ and $z\in \mathbb{R}^{n}$ that is uniformly distributed on unit $n$-sphere, we have
\[
\mathop{\mathbb{E}}(n\Vert Xz\Vert_2^2) = \Vert X\Vert_F^2.
\]
\end{theorem} 

In this paper, we fix $\mathcal{A}$ but vary $\mathcal{F}$ and $\mathcal{G}$, and learn multigrid smoothers that are appropriate 
for different PDEs from the same class. 
Specifically, we train the multigrid solvers on a small set of discretized problems 
\begin{equation} \label{eq:trainset}
\mathcal{D}=\bigcup_{j=1}^{Q}\{A,f_j, (u_{0})_j, (u_{\ast})_j\}
\end{equation}
with the presumption that the learned smoothers have good generalization properties: the learned smoothers can perform well on 
problems with much larger grid sizes and different geometries.

As a motivating example, we consider the following diffusion problem: 
\begin{equation} \label{eq:constant}
-\nabla \cdot(g\nabla u(x,y))=f(x,y),
\end{equation}
where $g$ is assumed to be constant in this section. We will consider the more general form $g(x,y)$ in the next section. 

Since the stencils for discretizing \eqref{eq:constant} would be identical for constant $g$ on structured meshes, the dynamics of the problems are spatial invariant and  independent of the specific location in the domain. Thus, we can parameterize the action of inverse of the smoother $(M^{(l)})\inv$ by one single convolutional neural network, $H^{(l)}$, with only convolutional layers. This parameterization has several advantages. First, on an $N\times N$ grid, $H^{(l)}$ only requires $O(N^2)$ computation and has a few parameters. Second, $H^{(l)}$ can be readily applied to problems defined on different grid sizes or geometries. Lastly, which is more important, \cref{thm:conv} justifies the use of this parameterization to construct convergent smoothers. 

\begin{theorem}\label{thm:conv}
%Suppose $H$ is a convolutional network with $k$ convolutional layers. For one fixed matrix $A$, if parameterized properly, $H$ can be a convergent smoother. 
For one fixed matrix $A$, there exists a finite sequence of convolution kernels $\{\omega^{(j)}\}_{j=1}^{J}$ such that the convolutional factorization $H=\omega^{(J)}*\dots\omega^{(2)}*\omega^{(1)}$ satisfies $\Vert I-HA\Vert_A<1$ indicating $H$ is a convergent smoother.
\end{theorem}
\begin{proof}
Based on the universality property of deep convolutional neural networks without fully connected layers \cite{ZHOU2020787}, we know that $H$ can approximate the linear operator $A^{-1}$ to an arbitrary accuracy measured by some norms when $k$ is large enough. Thus, theoretically, $HA$ can be very close to an identity mapping  if parameterized properly. Since all matrix norms are continuous and equivalent, $\Vert I-HA\Vert_A$ can be less than $1$ for certain $k$ measured in matrix $A$-norm.
\end{proof}
%we train the smoothers on the training dataset $\{(A\in \mathbb{R}^{N^{2}\times N^{2}},f_{j}\in\mathbb{R}^{N^{2}})\}_{j=1,\ldots,s}$ with one fixed small grid size on one square domain and different right-hand side vectors $f_{j}$. We show that the trained smoothers also work well on problems with large grid sizes and different geometries.

\subsection{Training and generalization}\label{learning strategy}
%$\mathcal{D}^{(l)}=\{(A^{(l)},f_j^{(l)},u_{0,j}^{(l)},u_j^{\ast(l)}\}_{j=1}^{Q}$ where $A^{(l)}$ is the coefficient matrix at level $l$, $u_{0,j}^{(l)}$ is a random initialization sampled from normal distribution and the ground truth solution $u_j^{\ast(l)}$ corresponding to the right-hand side vector $f_j^{(l)}$ is obtained by a direct solver. 
% For each problem instance in $\mathcal{D}^{(i)}$, it consists of the same matrix $A^{(i)}$, a right hand side $f^{(i)}$ and a initial guess $u_{0}^{(i)}$ which are sampled from normal distribution, a ground truth solution $u^{\ast(i)}$ which is solved exactly using direct solver (e.g. LU solver). 
In this section, we propose several strategies for training multigrid solvers using CNNs as smoothers. We will also discuss their advantages and disadvantages.

The first training strategy is to train $H^{(l)}$ separately  
for each multigrid level ${l=0, \dots, L-1}$, where
we construct a training set $\mathcal{D}^{(l)}$ similar to \cref{eq:trainset} for the operator $A^{(l)}$. That is, we train $H^{(l)}$ to make iteration~\cref{eq:iteration} convergent  by minimizing the error between the approximate solution obtained at iteration $k$ and the ground truth solution. As suggested in \cite{hsieh2018learning}, we also choose different iteration number $k$, $1 \le k \le b$ in the training, so that $H^{(l)}$ learns to converge at each iteration, where larger $b$ mimics the behavior of solving problems to 
higher accuracy while smaller $b$ mimics
inexpensive smoothing steps in multigrid.  

This  training strategy is simple and the trainings on different levels are totally independent.
However, we found the obtained $H^{(l)}$ usually do not exhibit good smoothing property of reducing high-frequency errors, especially when $H^{(l)}$ is a shallow neural network. This phenomenon is expected since the training strategy does not consider the underneath coarser-grid hierarchy and tries to reduce errors over
the whole spectrum of $A^{(l)}$.
In contrast, a well-trained $H^{(l)}$ with high complexities, deeper in the layers and larger in the convolution kernels, can approximate
the action of the inverse of $A^{(l)}$ well, but using it as a smoother is not efficient nonetheless, and moreover, the training cost will be significantly higher.

A second training approach is to optimize the objective function \cref{eq:overalltrain} directly over $M^{(l)}$ at all levels, $l=0,\ldots,L-1$. 
This approach targets at optimizing convergence of the overall multigrid V-cycles and considers both the smoothing and the coarse-grid correction.
However,  training the CNNs at all levels together turns out to be prohibitively expensive.

Finally, we propose an efficient adaptive training strategy that can impose the smoothing property by recursively 
training the smoothing CNN at a fine level. %adaptively learn the smoothing property of smoothers among levels during training. Instead of training the smoothers as solvers to decrease the error at each level, we train the smoothers as a part of the multigrid method which also contains other smoothers so that at each level, the smoother can learn to eliminate the the high-frequency errors and pass the low-frequency errors to the coarse grid which will be solved by a multigrid solver equipped with lower-level smoothers.
%This is done in a reverse order. 
%We start with 
The training process starts from the second coarsest level and is repeatedly applied to the finer levels, given that
the smoothers at coarser levels have been already trained, so that solve with the coarse-grid operator can be replaced with
a V-cycle using the available multigrid hierarchy at one level down.
The adaptive training algorithm is sketched in ~\cref{alg:adap-train}. 

\begin{algorithm}[tb]
   \caption{Adaptive training of multigrid CNN smoothers}
   \label{alg:adap-train}
\begin{algorithmic}[1]
   \STATE {\bfseries Input:} Multigrid hierarchy: number of multigrid levels $L+1$, coarsest-grid solver at level $L$, namely $\Psi^{(L)}$, coefficient matrix $A^{(l)}$ where $A^{(0)}=A$, and interpolation operator $P^{(l)}$, for $l=0, \ldots, L-1$. Size of  training set $Q$. Maximum allowed number of smoothing steps $b$
   \STATE {\bf Output:} Smoothers $H^{(0)},\dots H^{(L-1)}$
   \FOR{$l=L-1,\dots,0$}
 %  \STAhgze previous smoothers $H^{(L-1)},\dots,H^{(i+1)}$
   \STATE Construct training set: 
   \[\mathcal{D}^{(l)}=\bigcup_{j=1}^{Q} \{t_j\}, \quad t_j^{(l)}= \{A^{(l)},f^{(l)}_j, (u^{(l)}_{0})_j, (u^{(l)}_\ast)_j\} \]
%   \STATE Construct solver $\Phi^{(l)}$ which is parameterized by $H^{(l)}$:
%    \STATE    $\Phi^{(l)}(u^{(l)}_{0},f^{(l)},1)=(I-H^{(l)}A^{(l)})(u^{(l)}+P^{(l)}\delta)$ where 
%    
%    $\delta = \Phi^{(l+1)}(0,(P^{(l)})^{T}(f^{(l)}-A^{(l)}u^{(l)}),1)$
%    and 
%     $u^{(l)} = (I-H^{(l)}A^{(l)})u^{(l)}_{0}-H^{(l)}f^{(l)}$
    \STATE Initialize the weights of $H^{(l)}$
   \STATE Perform stochastic gradient descent (SGD) to minimize  loss function: 
   \begin{align} \notag
%    \sum\limits_{\substack{t_j^{(l)}\in\mathcal{D}^{(l)}\\ k\sim\mathcal{U}(1,b)}}
\sum_{{t_j^{(l)}\in\mathcal{D}^{(l)}, k\sim\mathcal{U}(1,b)}}
   \Vert \Phi^{(l)}((u^{(l)}_{0})_j, f_j^{(l)}, k) - (u^{(l)}_\ast)_j \Vert_2
   \end{align}
   With $\Phi(u_0, f, 0) \equiv u_0$, run forward propagation by
   \begin{align*} 
    \Phi^{(l)}(u_0, f, k) &= \Phi^{(l)}(u_0, f, k-1) + \Psi^{(l)}(r_{k-1}), \\
    r_{k-1} &:= f - A^{(l)} \Phi(u_0, f, k-1), \\
    \Psi^{(l)}(r_{k-1}) &= t_{k-1} + H^{(l)}(r_{k-1}-At_{k-1}), \\
    t_{k-1} &:= H^{(l)}(r_{k-1}) + P^{(l)} \Psi^{(l+1)} ((P^{(l)})\trans s_{k-1}), \\
    s_{k-1} &:= r_{k-1} - A H^{(l)}(r_{k-1}),
   \end{align*}

%    \begin{enumerate}[label=(\alph*)]
%     \item a
%    \end{enumerate}
    and update $H^{(l)}$ by back propagation
   \ENDFOR
    
\end{algorithmic}
\end{algorithm}

\begin{figure}[h]
  \centering
  %\resizebox{\width}{5cm}
  {\includegraphics[height=0.36\linewidth, width=0.9\linewidth]{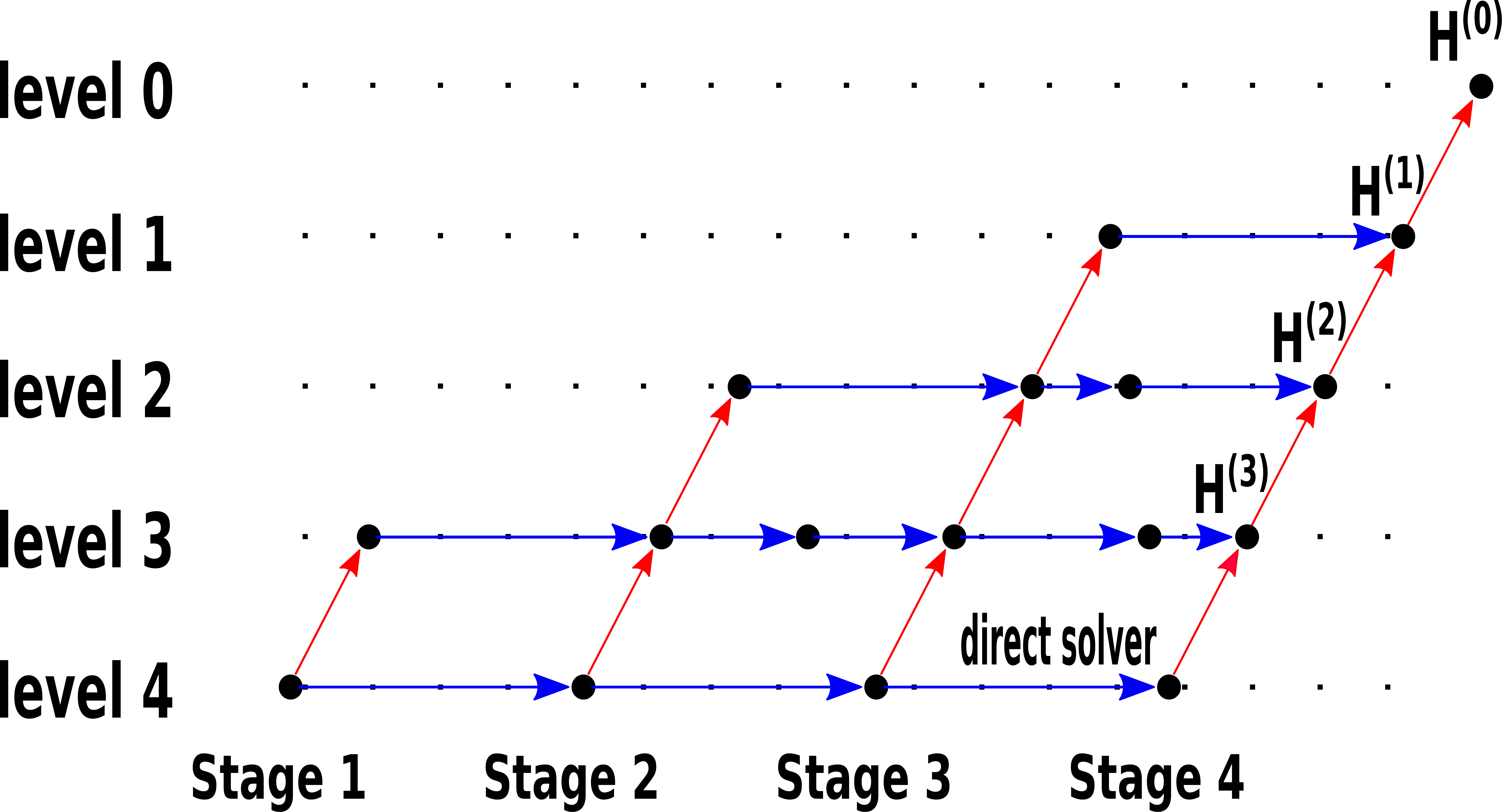}}
  \caption{Proposed adaptive training strategy for a $5$-grid method. The training starts from level $3$ and proceeds upward to level $0$.  When $H^{(l)}$ is being trained, lower level $H^{(j)}$ for $j=l+1,\ldots, 3$ are used in the solve with the coarse-grid operator and remain unchanged.}
  \label{framework}
\end{figure}

%Since the learned smoothers are constructed by a stack of convolutional layers which can gather the information of the neighborhood around each grid point to smooth the solution, they are not restricted to a certain grid size or geometry. 

~\cref{framework} illustrates the procedure of adaptively training a 5-level multigrid solver in 4 stages, 
starting at level~$3$. The loss is given by 
$$L^{(3)}=\sum_{j,k}\norm{\Phi^{(3)}((u_{0}^{(3)})_j,(f^{(3)})_j,k)-(u_{\ast}^{(3)})_j}_2,$$
where $\Phi^{(3)}$ represents the two-level multigrid with levels $3$ and $4$.
In the second stage, the training proceeds at level $2$ for CNN $H^{(2)}$ utilizing the underlying 2-level hierarchy obtained from the first stage.
This procedure continues until  $H^{(0)}$ is computed at the finest level and the entire training is completed, so the resulting multigrid hierarchy
$\Phi^{(0)}$ can be used for solving systems of equations with $A^{(0)}\equiv A$.

Another appealing property of the proposed training approach is the updatability of smoothers using neural networks.
The trained smoothers can be updated in another training process by
injecting the errors that cannot be effectively reduced by the current multigrid solver back to the training set.
Specifically, to improve the smoothers in a trained multigrid solver $\Phi^{(0)}$, we can first apply $\Phi^{(0)}$ 
to homogeneous equation $Au=0$ 
for $k$ steps with a random initial vector  $u_0$ 
and get the approximate solution $u_k$, i.e., $u_k=\Phi^{(0)}(u_{0},0,k)$, then inject the (restricted) residual,
$r_k^{(l)}=(P^{(l-1)})\trans r_k^{(l-1)}$ with $r_k^{(0)}=-Au_k$
to the training set at each level $l$, and finally re-train $\Phi^{(0)}$  as before with the new augmented training sets 
using the existing $H^{(l)}$ in the multigrid hierarchy as the initial values.

\section{Learning deep neural smoothers for variable coefficient PDEs}
\label{sec:VariableCoeff}
In this section, we extend the adaptive training framework proposed in Section \ref{Learning} to design optimal smoothers for solving variable coefficient PDEs:
\begin{equation}
\label{eq:vcoeff}
-\nabla \cdot(g(x,y)\nabla u(x,y))=f(x,y).
\end{equation}

To better illustrate the difficulty of dealing with variable coefficient PDEs, we simplify our discussion and consider discretizing \eqref{eq:vcoeff} using nine-point stencils with grid spacing $h$. See the left subfigure of \cref{stencil} for a demonstration of $3\times 3$ neighborhood of the grid point $u_{22}$. The equation corresponds to the grid point $u_{22}$ reads:
\begin{align*}
    &-\frac{1}{3h^{2}}(g_{1}u_{11}+g_{2}u_{13}+g_{3}u_{31}+g_{4}u_{33})\\
    &-\frac{1}{6h^{2}}((g_{1}+g_{2})u_{12}+(g_{2}+g_{4})u_{23}+(g_{3}+g_{4})u_{32}+(g_{1}+g_{3})u_{21})\\
    &+\frac{2}{3h^{2}}(g_{1}+g_{2}+g_{3}+g_{4})u_{22}=f_{22}.
\end{align*}

\begin{figure}
    \centering
    \begin{subfigure}[b]{0.35\columnwidth}
    \centering
    \caption{Weight stencils}
    \includegraphics[width=0.6\columnwidth]{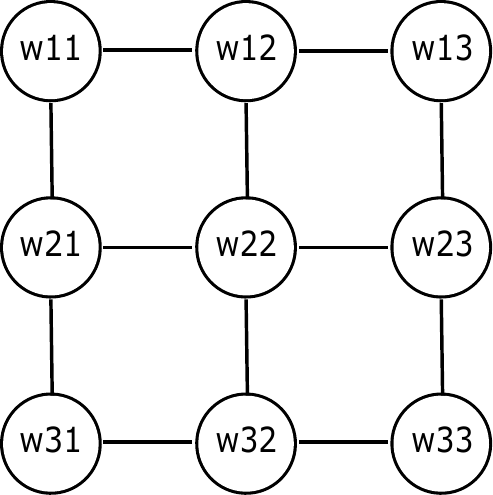}
    \end{subfigure}
    %
    %\hspace{1em}
    %
    \begin{subfigure}[b]{0.35\columnwidth}
    \centering
    \caption{Grid points}
    \includegraphics[width=0.6\columnwidth]{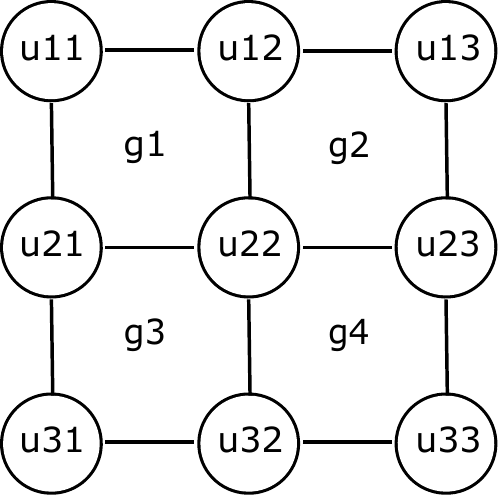}
    \end{subfigure}
    \caption{The architecture of inferring the smoothing kernels for the central point in the stencil. A fully connected neural network takes nine $3\times 3$ stencils as input and outputs three convolution smoothing kernels for the current grid point.}
    \label{stencil}
\end{figure}

This is equivalent to applying a $3\times 3$ weight stencil to the $3\times 3$ neighborhood of $u_{22}$ as 
$$\sum\limits_{i,j}w_{ij}u_{ij}=f_{22},\quad  \ i, j \in\{0,1,2\},$$ where $w_{ij}$ are computed according to the function $g(x,y)$ and is shown in the right subfigure of \cref{stencil}.
When $g(x,y)$ is constant, the coefficients $w_{ij}$ correspond to each interior $3\times 3$ stencil are identical. Thus, we can parameterize $M^{-1}$ by a single convolutional neural network as a stack of convolution kernels $\{\phi_{i}\}$. The weights of each convolution kernel $\phi_{i}$ are shared over all grid points.
However, when $g(x,y)$ is variant, the weight stencils $W_{ij}$ and $W_{lm}$ at two different locations can have completely different dynamics (e.g. $W_{ij}$ can be strong in $x$-axis and weak in $y$-axis while $W_{lm}$ is strong in $y$-axis and weak in $x$-axis). In this case, a smoothing kernel $H$ that is learned to smooth the error at one grid point might be ineffective in smoothing the error at another point. As a result, the optimal smoothing kernel $H_{ij}$ associated with each grid point should be conditioned on the location for variable coefficient problems.

In order to generate unshared convolution kernels which are dimension-invariant, we propose to learn a function which can adaptively adjust the kernels based on the spatial information. In particular, we will design neural network architectures which can map each grid representation to a stack of convolution kernels that can be used to efficiently smooth the error at different locations. 

\subsection{Parameterization with fully connected layers}
In the first approach, we consider using multiple layer perceptron to construct the mapping from the grid representation to the smoothing kernels at each grid point.  Although the stencil at each grid point has already contained the spatial information, we find that only using the stencil information as the representation is not sufficient to learn efficient smoothing kernels and the generalization usually performs poorly. Instead, we suggest to incorporate the neighborhood information into the grid representation. More specifically, we construct each grid representation as an $81\times 1$ vector which consists of the stencils in the $3\times 3$ neighborhood of the current point under consideration. In this case, the feature map $\mathbb{M}$ for an $N\times N$ grid has the size of ${N^{2}\times 81}$.  The mapping is then parameterized by a fully connected neural network which takes the representation of each grid point as input and infers the weights of the $k$ output smoothing kernels of size $3\times 3$.  
See \cref{fig:hypernet} for an illustration of this architecture. To smooth the error at the central point in the stencil, we train a fully connected neural network which takes nine $3\times 3$ stencils with $81$ parameters in total and outputs three $3\times 3$ convolution kernels that are used to smooth the error at this point. On each level of multigrid solver, we only construct one such neural network based on the adaptive training strategy  discussed in Section \ref{learning strategy}.
\begin{figure}
    \centering
    \includegraphics[width=0.5\columnwidth]{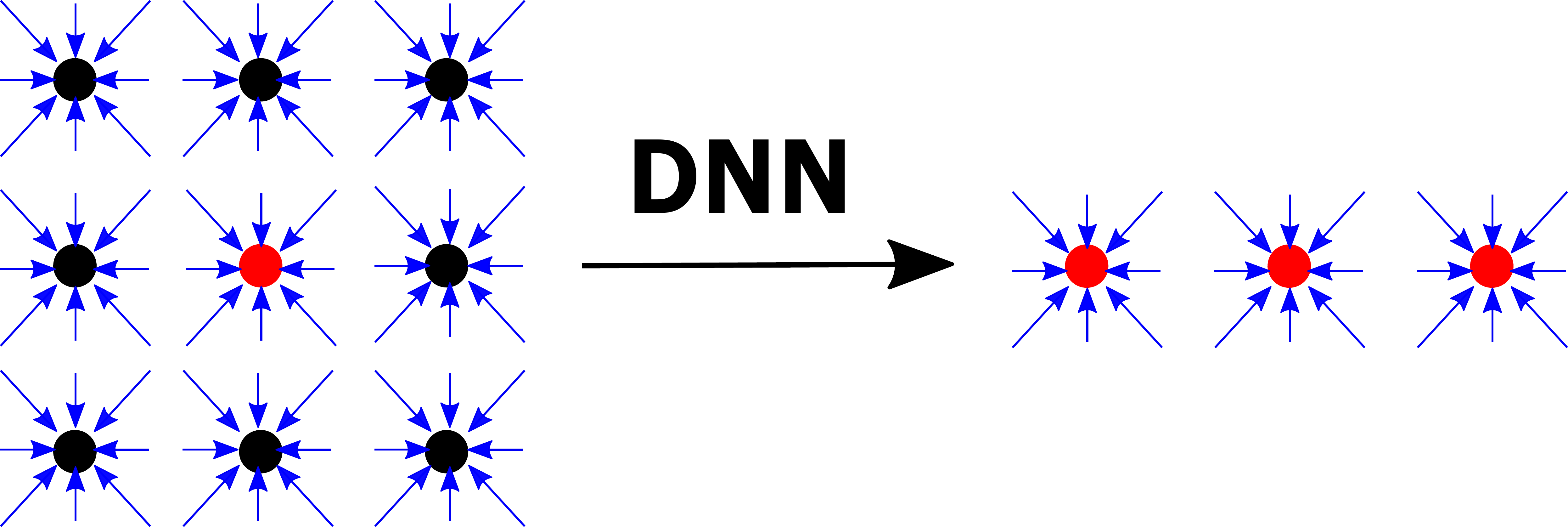}
    \caption{The architecture of inferring the smoothing kernels for the central point in the stencil. A fully connected neural network takes nine $3\times 3$ stencils as input and outputs three convolution smoothing kernels for the current grid point.}
    \label{fig:hypernet}
\end{figure}

    % $\begin{bmatrix}&T_{i-1,j}&\\T_{i,j-1}&T_{i,j}&T_{i,j+1}\\&T_{i+1,j}&\end{bmatrix}$ as input and output $H_{i,j}$
\subsection{Parameterization with convolutional layers}
Deep neural networks using fully connected layers often require a large amount of parameters in order to well approximate a function and also have high training cost. In order to reduce the training cost, instead of constructing a feature map $\mathbb{M}\in \mathbb{R}^{N^{2}\times 81}$ by flattening and stacking the stencils  and applying fully connected neural networks, an alternative approach is to feed into the neural network with 9 channels with each channel corresponding to one stencil in the $3\times 3$ neighborhood of the point under consideration. The deep neural network is parameterized by several convolution kernels followed by a fully connected layer. The outputs of the neural network are $k$ smoothing kernels. This architecture is illustrated in \cref{fig: convolution}. We will show in numerical experiments that this approach can achieve a comparable performance with fully connected layers but requires much fewer parameters.

\begin{figure}
    \centering
    \includegraphics[width=0.8\columnwidth]{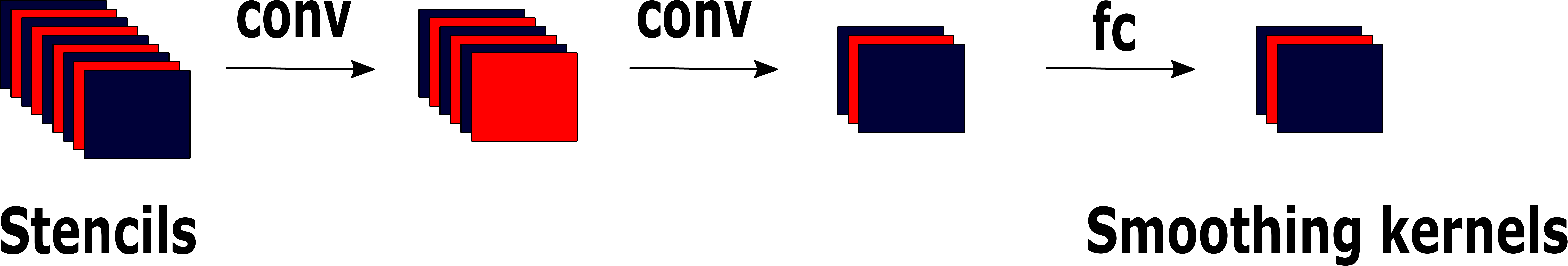}
    \caption{The framework of constructing 3 smoothing kernels by applying two convolutional layers and one fully connected layer to a feature map. The feature maps have $9$,$6$,$3$ and $3$ channels, respectively and each channel contains kernels of size $3\times 3$.}
    \label{fig: convolution}
\end{figure}

% Consider $-\nabla \cdot (T(x,y)\nabla u(x,y)) = f(x,y)$

\section{Interpretation of learned smoothers}\label{inter}
In this section we present the patterns of the learned smoothing kernels. We consider the anisotropic rotated Laplacian problem \eqref{pde} parameterized by the angle $\theta$ of the anisotropy and conductivity $\xi$. We fix $\xi=100$ and train smoothers for problems with a variety of $\theta \in \{0,\frac{\pi}{12},\frac{\pi}{6},\frac{\pi}{4},\frac{\pi}{3},\frac{5\pi}{12},\frac{\pi}{2}\}$. For each problem, we use a two-grid solver and on the fine level we train a smoother which consists of one convolution kernel of size $9\times 9$. We use linear activation in order to illustrate the action of the convolution kernels as the smoothers. The trained convolution kernels corresponding to different $\theta$ are shown in \cref{fig:patterns1}. The results show that large values in each kernel are gathered symmetrically about the center and the angles of the large values of each kernel also align with the angle of the anisotropy of the problem. These patterns demonstrate that the learned smoothing kernels are able to smooth the error
in correct directions, which can be viewed as line smoothers truncated in the convolution windows along the direction of strong couplings.
%in the most relevant regions.

We also increase the number of convolutional layers and study the impact of each convolutional layer on the final smoother. For each problem we train three convolution kernels of size $9\times 9$ and show the results in \cref{fig:patterns2}. The first row shows the kernels of the first convolutional layer for each problem while the second row and the third row show the second layer and the third layer respectively. The kernels at different layers exhibit different patterns which indicates that each kernel is responsible for smoothing the error in different regions.  Since applying three $9\times 9$ convolution kernels sequentially is equivalent to applying a $25\times 25$ convolution kernel, we illustrate the patterns of the effective  $25\times 25$ kernels in the last row of \cref{fig:patterns2}. The kernels in the last row display similar patterns as in \cref{fig:patterns1} which perfectly align with the anisotropy of the problem.
\begin{figure}[h]
% \vskip -0.1in
\begin{center}
\includegraphics[width=0.13\columnwidth]{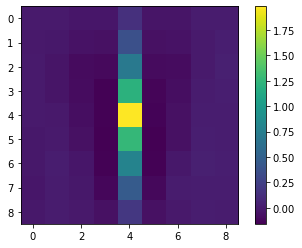}
\includegraphics[width=0.13\columnwidth]{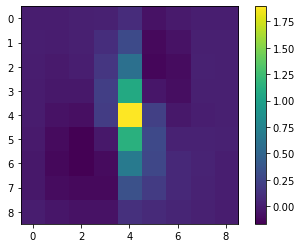}
\includegraphics[width=0.13\columnwidth]{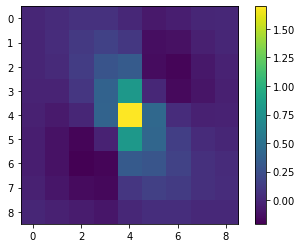}
\includegraphics[width=0.13\columnwidth]{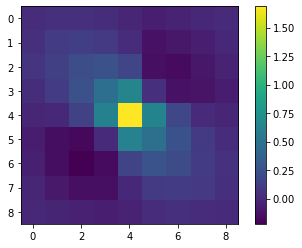}
\includegraphics[width=0.13\columnwidth]{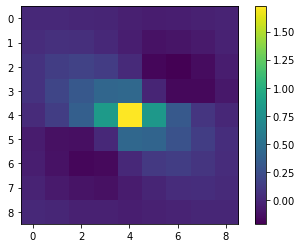}
\includegraphics[width=0.13\columnwidth]{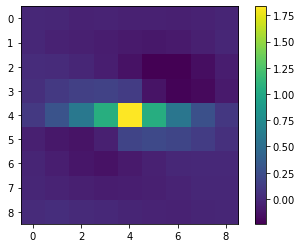}
\includegraphics[width=0.13\columnwidth]{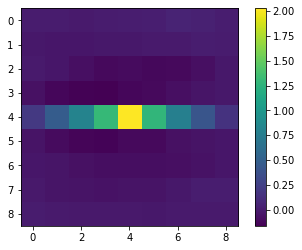}
   \caption{The patterns of the trained kernels on problems where $\xi=100$ and $\theta \in \{ 0,\frac{\pi}{12},\frac{\pi}{6},\frac{\pi}{4},\frac{\pi}{3},\frac{5\pi}{12},\frac{\pi}{2} \}$. For each problem, the smoothers are trained using only one kernel.
}
\label{fig:patterns1}
\end{center}
% \vskip -0.2in
\end{figure}

\begin{figure}[h]
% \vskip -0.1in
\begin{center}
\includegraphics[width=0.13\columnwidth]{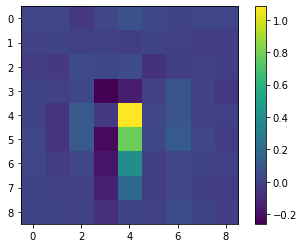}
\includegraphics[width=0.13\columnwidth]{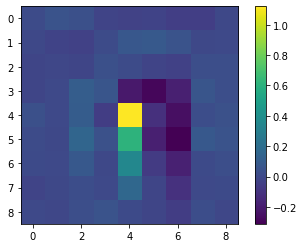}
\includegraphics[width=0.13\columnwidth]{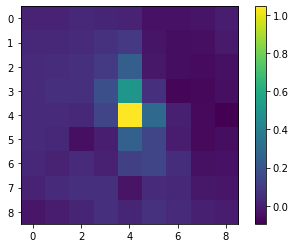}
\includegraphics[width=0.13\columnwidth]{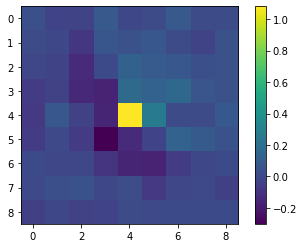}
\includegraphics[width=0.13\columnwidth]{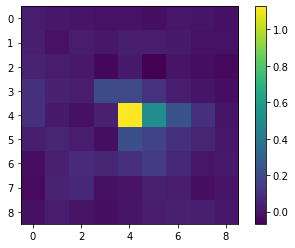}
\includegraphics[width=0.13\columnwidth]{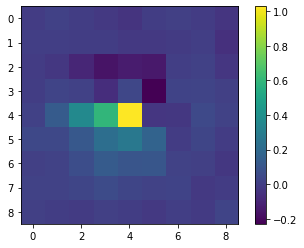}
\includegraphics[width=0.13\columnwidth]{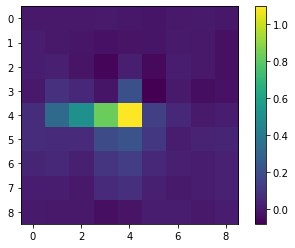}

\includegraphics[width=0.13\columnwidth]{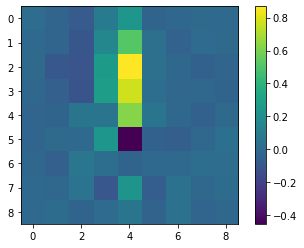}
\includegraphics[width=0.13\columnwidth]{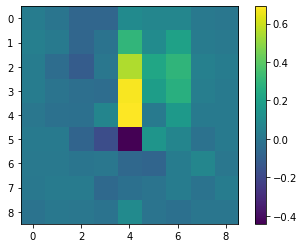}
\includegraphics[width=0.13\columnwidth]{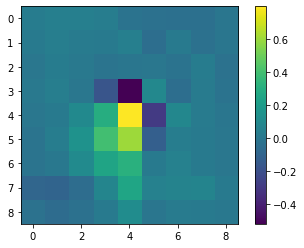}
\includegraphics[width=0.13\columnwidth]{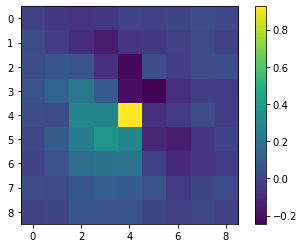}
\includegraphics[width=0.13\columnwidth]{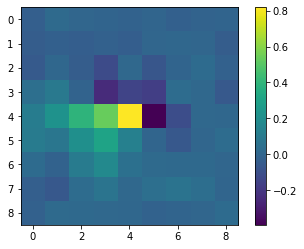}
\includegraphics[width=0.13\columnwidth]{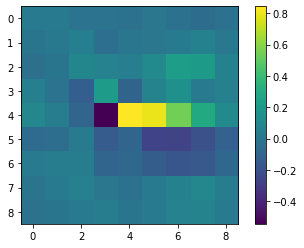}
\includegraphics[width=0.13\columnwidth]{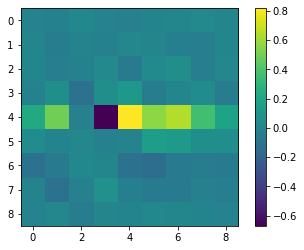}

\includegraphics[width=0.13\columnwidth]{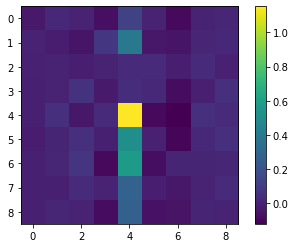}
\includegraphics[width=0.13\columnwidth]{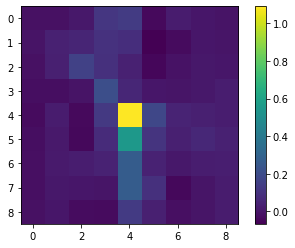}
\includegraphics[width=0.13\columnwidth]{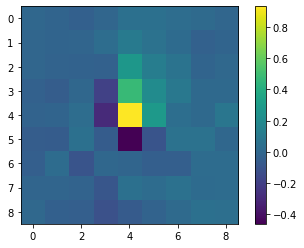}
\includegraphics[width=0.13\columnwidth]{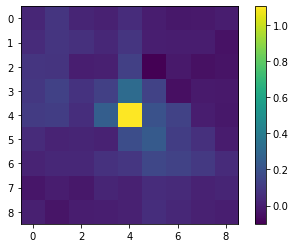}
\includegraphics[width=0.13\columnwidth]{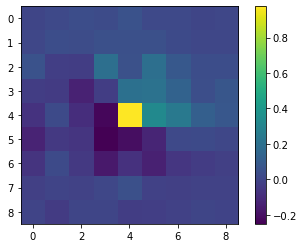}
\includegraphics[width=0.13\columnwidth]{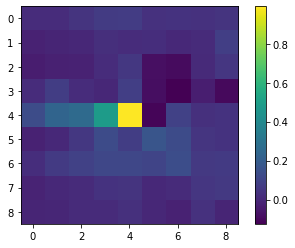}
\includegraphics[width=0.13\columnwidth]{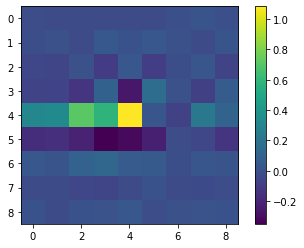}

\includegraphics[width=0.13\columnwidth]{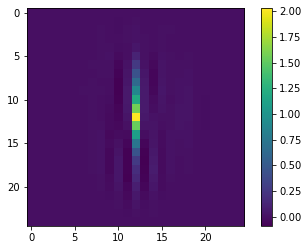}
\includegraphics[width=0.13\columnwidth]{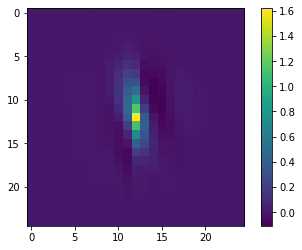}
\includegraphics[width=0.13\columnwidth]{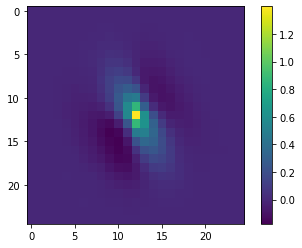}
\includegraphics[width=0.13\columnwidth]{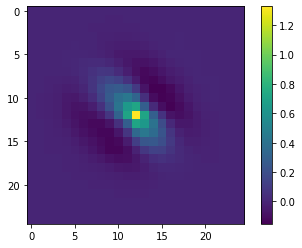}
\includegraphics[width=0.13\columnwidth]{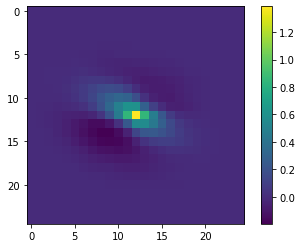}
\includegraphics[width=0.13\columnwidth]{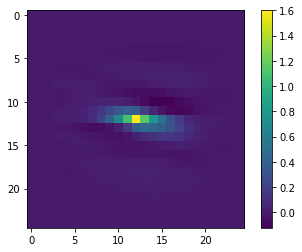}
\includegraphics[width=0.13\columnwidth]{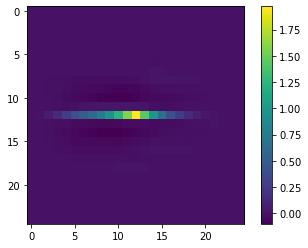}
   \caption{The patterns of the trained kernels on the problem \eqref{pde} where $\xi=100$ and $\theta \in \{ 0,\frac{\pi}{12},\frac{\pi}{6},\frac{\pi}{4},\frac{\pi}{3},\frac{5\pi}{12},\frac{\pi}{2}\}$. For each problem, the smoothers are trained using three kernels. The first three rows represent the kernels on the first layer, second layer and third layer respectively and the last row combines the three kernels into one single kernel for each problem.
}

\end{center}
\label{fig:patterns2}
% \vskip -0.2in
\end{figure}

\section{Numerical experiments}

\label{sec:exp}
% \textbf{Experiment set up}
In this section, we provide numerical examples to demonstrate the smoothing effect of the proposed smoothers. All of the codes were implemented in PyTorch 1.8.1  and run on an Intel Core i7-6700 CPU. We use a batch size of 10 and employ the Adam optimizer 
with a learning rate of $10^{-3}$ for 500 epochs. The neural network training took roughly 5 hours for each constant coefficient problem and roughly 4 hours for each variable coefficient problem.
\subsection{Constant coefficient PDEs}
We first consider the following two dimensional  anisotropic rotated Laplacian problem:
\begin{equation}\label{pde}
    -\nabla \cdot (T\nabla u(x,y)) = f(x,y) %\quad  (x,y) \in \mathcal{G}
\end{equation}
where $2\times 2$ tensor field $T$ is defined as
\begin{equation}
T=\begin{bmatrix}
\cos^{2}{\theta}+\xi\sin^{2}{\theta}&\cos{\theta}\sin{\theta}(1-\xi) \\ 
\cos{\theta}\sin{\theta}(1-\xi) & \sin^{2}{\theta}+\xi\cos^{2}{\theta}
\end{bmatrix},
\end{equation}
with $\theta$ being the angle of the anisotropy and $\xi$ being the conductivity.
We discretize the operators $\Delta u$ and $u_{xy}$ in \cref{pde} using the following  stencils: %five-point
\begin{equation} \notag
\frac{1}{4h^{2}}\begin{bmatrix}&-1 &\\-1&4&-1\\&-1&\end{bmatrix} \quad \mbox{and}\quad 
\frac{1}{2h^{2}}\begin{bmatrix}&-1&1\\-1&2&-1\\1&-1&\end{bmatrix},
\end{equation}
where $h$ is the grid spacing.
%\begin{equation} \notag
%$\frac{1}{4h^{2}}\begin{bsmallmatrix}&-1 &\\-1&4&-1\\&-1&\end{bsmallmatrix}$ and
%$\frac{1}{2h^{2}}\begin{bsmallmatrix}&-1&1\\-1&2&-1\\1&-1&\end{bsmallmatrix}$,
%\end{equation}
%where $h$ is the grid spacing.

We  use multigrid V-cycles to solve the resulting discretized linear system $Au=f$, where the coefficient matrix $A$ is parameterized with ($\theta$, $\xi$, $n$, $G$).  Here $n$ is the grid size and $G$ is the geometry of the grid.

We show the robustness and efficiency of the proposed neural smoothers on a variety of sets of parameters $(\theta,\xi,n,G)$. For each set of the parameters, 
we train the neural smoothers on dataset constructed on square domains with small grid size,
and show that the trained neural smoothers can outperform standard ones such as weighted Jacobi. Furthermore, we demonstrate that the trained neural smoothers can be applied to solve much larger problems and problems with more complex geometries without retraining.

Since our focus of this work is on smoothers, we adopt standard algorithms for multigrid coarsening and grid-transfer operators.
Specifically, we consider full coarsening, which is illustrated in \cref{fig:full_coarsening} for 2D grids, where
grid points are coarsened in both $x$- and $y$-dimensions.
The associated restriction and interpolation  are full weighting, a weighted average in $3\times 3$ neighborhood. 
The stencils of the restriction and interpolation operators are given by, respectively,
$$
\frac{1}{16}
\left[
\begin{array}{ccc}
1&2&1\\
2&4&2\\
1&2&1\\
\end{array}
\right]
\quad \mbox{and} \quad
\frac{1}{4}
\left]
\begin{array}{ccc}
1&2&1\\
2&4&2\\
1&2&1\\
\end{array}
\right[. 
$$

\begin{figure}[ht!]
\caption{Full coarsening of a 2D grid. Fine points are red and  coarse points are black.
Full weighting restriction is shown by the arrows to the coarse point at the center.
\label{fig:full_coarsening}}
\vspace{1cm}
\centering\includegraphics[width=0.4\columnwidth]{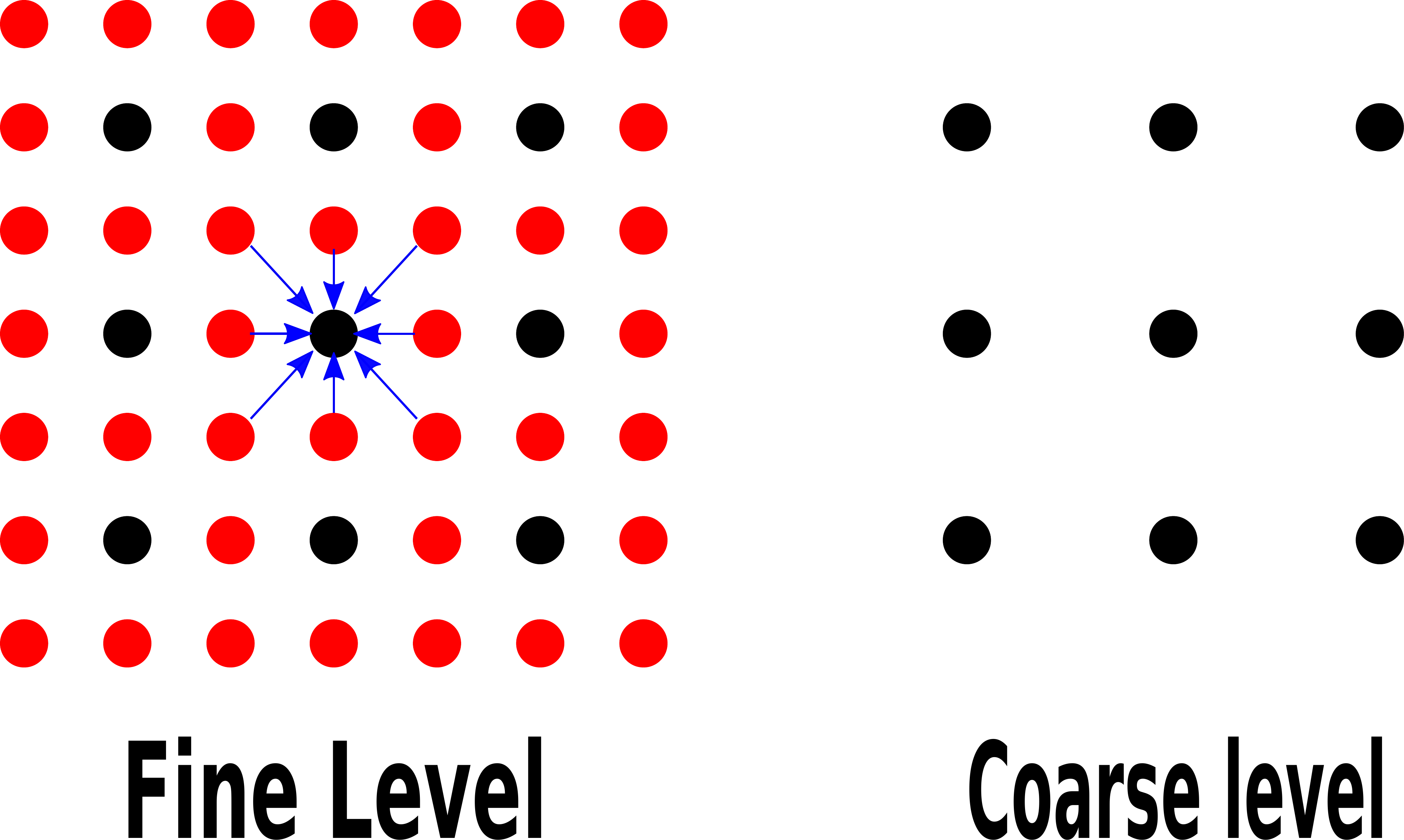}
\end{figure}

We also consider red-black coarsening that has a coarsening factor of about $2$ shown in \cref{fig:rb_coarsening} for 
the first $3$ levels. Note that the coarsening on level $0$ is essentially a semi-coarsening along the $45^{\circ}$ angle, and on level $1$
the coarsening is performed on the $45^{\circ}$-rotated meshes, which generates the grid on level $2$ that  amounts to a semi-coarsening
along the $y$-dimension.
The restriction and interpolation stencils used associated with this coarsening are given by 
$$
\frac{1}{8}
\left[
\begin{array}{ccc}
&1&\\
1&4&1\\
&1&\\
\end{array}
\right]
\quad \mbox{and} \quad
\frac{1}{4}
\left]
\begin{array}{ccc}
&1&\\
1&4&1\\
&1&\\
\end{array}
\right[
.$$

\begin{figure}[h]
\caption{Red-black coarsening of a 2D grid, where  red points are  fine points and  black points are  coarse points.
\label{fig:rb_coarsening}}
\vspace{1cm}
\centering\includegraphics[width=0.7\columnwidth]{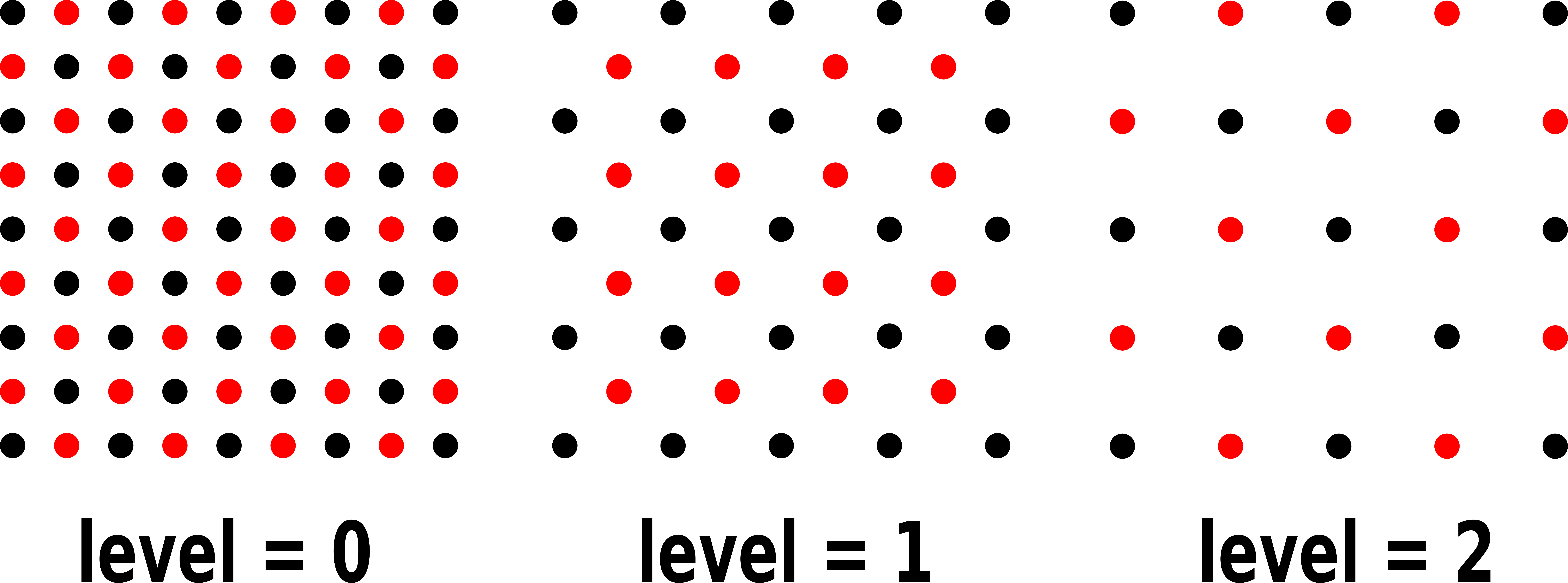}
\end{figure}

To evaluate our method, we compare the performance of multigrid methods using \cref{alg:multigrid} equipped with  convolutional neural smoothers  that are trained adaptively   (denoted by \textit{$\alpha$-CNN}), convolutional neural smoothers trained independently (denoted by \textit{CNN}) and weighted Jacobi smoother (denoted by \textit{$\omega$-Jacobi }) for solving a variety of linear systems. These problems are generated by varying the parameters ($\xi$, $\theta$,  $n$,  $G$). The weight $\omega$ is chosen to be $\frac{2}{3}$ by heuristics for all experiments in this paper.

\paragraph{Training details}
First, we train smoothers independently using the first strategy discussed in Section \ref{learning strategy}.  For each smoother, we construct $50$ problem 
instances of size $16^{2}$. 
Then, we use the adaptive training framework to train smoothers using \cref{alg:adap-train}.  
The training process for a 5-level multigrid has 4 stages. At each 
stage we construct a training data set which contains $50$ instances of the problem on each level. 
All stages have the same size of the coarsest grid. 
In particular, under full coarsening scheme, at stage $l$ the problems are constructed on the $(4-l)$th level and have grid size of $(2^{l+2}-1)^{2}$. 
Under red-black coarsening  scheme, at stage $1$ and stage $2$ the problem instances have size of $9^{2}$ and at stage $3$ and stage $4$ the problems have size of $17^{2}$. 
This is because when we apply 
red-black coarsening to a regular grid, the grid becomes irregular, 
therefore we need to add zeros to the irregular grid so that we can apply CNNs more efficiently.

\paragraph{Neural networks}

We use CNNs to approximate the action of the  inverse of the smoothers.
In particular,
under full coarsening scheme, for both CNN  and $\alpha$-CNN smoothers,   $H^{(l)}$ is parameterized as follows:
\begin{equation}\label{full-coarsening-smoother}
    H^{(l)}=f^{(l)}_{5}(f^{(l)}_4(\cdots (f^{(l)}_{2}(f^{(l)}_{1}))\cdots)+f^{(l)}_{6},
\end{equation}
where each $f^{(l)}_{i}$ is parameterized by a $3\times 3$ convolution kernel $\phi^{(l)}_{i}$. We initialize the weights of $\phi^{(l)}_{1},\dots,\phi^{(l)}_{5}$ with zeros  and  the weights of $\phi^{(l)}_{6}$ to be the inverse Jacobi stencil so that $H^{(l)}$ is initialized as  Jacobi. 
For red-black coarsening,  $H^{(l)}$ is parameterized as 
\begin{equation}\label{red-black-smoother}
    H^{(l)}= f^{(l)}_{2}(f^{(l)}_{1}).
\end{equation}
Note that we could use more convolutional layers and for each grid point we could also 
explore a larger range of the neighborhood, which can typically 
lead to a faster convergence rate at the price of more computational costs per iteration.
The current settings are found to give the best trade-off between convergence rate and time-to-solution.

\paragraph{Evaluation metrics}

We train the smoothers on problems with small grid sizes where the ground truth can be easily obtained. When we test on large-scale problems, it is time consuming to obtain the ground truth. Therefore when we evaluate the performance, we use the convergence threshold relative residual $\frac{\Vert f-A\hat{u}\Vert_2}{\Vert f\Vert_2}<10^{-6}$ as the stopping criterion which can avoid the requirement of exact solutions. 
We compare both the number of iterations and the runtime for multigrid solvers using different smoothers to reach the same accuracy. 
To reduce the effect of randomness, for each test problem, we run the multigrid solvers to solve $10$ problems with different random right-hand sides 
and present the averaged numbers.

\paragraph{Convergence rate}

Since coarser problems are usually better conditioned, the smoothers on the finest level have the biggest impact on the overall convergence. In this experiment we compare the spectral properties of the smoothers on the finest level.
We first compare the spectral radius of the iteration matrices \cref{eq:iteration} constructed by $\omega$-Jacobi smoothers  ($\omega$ is fixed at $\frac{2}{3}$ in all experiments) and $\alpha$-CNN smoothers and summarize the results in \cref{spectral-radius-1}.  These statistics are calculated on two sets of test problems defined on one $16\times 16$ grid.  In the first set,  $\theta$ is fixed as $0$ and $\xi=100,200,300,400$. In the second set,  $\xi$ is fixed at $100$ and $\theta = 0,\frac{\pi}{12},\frac{\pi}{6},\frac{\pi}{4}$. 
The corresponding comparison of ideal convergence bounds \cref{eq:beta_star} on these tests is provided in \cref{min-weak-approx}. 

\begin{table}[!htbp]
\centering
\centering
%\scalebox{0.9}
{
\begin{tabular}[t]{c|c|c|c|c}
\centering
% \hline
% \multicolumn{5}{|c|}{We fix $\theta=0$}\\\hline
$\theta=0$ &  $\xi=100$&   $\xi=200$ &  $\xi=300$&$\xi=400$\\\hline
$\omega$-Jacobi &  0.9886 & 0.9886   & 0.9886    & 0.9886 \\%\hline  %   \\ \hline
Gauss-Seidel& 0.9662 & 0.9662 &0.9662 &0.9662  \\\hline
$\alpha$-CNN & 0.7660  & 0.8060 & 0.8588  & 0.7883     %\\
% \end{tabular}}
% \scalebox{0.95}{\begin{tabular}[t]{|c|c|c|c|c|}
% \hline\hline
% $\xi=100$  & $\theta=0$&$\theta=\frac{\pi}{12}$&$\theta=\frac{\pi}{6}$&$\theta=\frac{\pi}{4}$\\ \hline
% $\omega$-Jacobi &0.9853  & 0.9436 &0.8981 &0.8837\\ %\hline
% % Gauss-Seidel&0.9002  &0.8634 &0.8682   & 0.9564 & 0.8566 &0.7776&0.7643 \\\hline
% $\alpha$-CNN& 0.6816 &0.4534  &0.4547 &0.4216\\%\hline
\end{tabular}

\vspace{1em}
\centering
\begin{tabular}[t]{c|c|c|c|c}
% \hline
% \multicolumn{5}{|c|}{We fix $\theta=0$}\\\hline
% $\theta=0$ &  $\xi=100$&   $\xi=200$ &  $\xi=300$&$\xi=400$\\\hline
% $\omega$-Jacobi &0.9853  &0.9918   &0.9940    &0.9951\\%\hline  %   \\ \hline
% % Gauss-Seidel&0.9564  &0.9755 &0.9820   & 0.9853 & 0.9873  \\\hline
% $\alpha$-CNN & 0.6816 &0.8189 & 0.8805  &0.8936     \\
% % \end{tabular}}
% % \scalebox{0.95}{\begin{tabular}[t]{|c|c|c|c|c|}
% \hline\hline
$\xi=100$  & $\theta=0$&$\theta={\pi}/{12}$&$\theta={\pi}/{6}$&$\theta={\pi}/{4}$\\ \hline
$\omega$-Jacobi& 0.9886 &0.9913  & 0.9934  & 0.9942   \\ %\hline
Gauss-Seidel&0.9662 &0.9735 &0.9797   & 0.9823\\
$\alpha$-CNN&  0.7660  & 0.7743  &  0.9652  & 0.9728  \\%\hline
\end{tabular}
}
\caption{Spectral radius of iteration matrices \cref{eq:iteration} of two-grid methods using full coarsening and $\omega$-Jacobi with $\omega=\frac{2}{3}$, Gauss-Seidel and $6$-layered $\alpha$-CNN smoothers 
for rotated Laplacian problems with different $\theta$ and $\xi$.
The  grid size is $16\times 16$.}
\label{spectral-radius-1}
\end{table}

\begin{table}[!htbp]
%\scalebox{1}
%{
\centering
\begin{tabular}[t]{c|c|c|c|c}
% \hline
% \multicolumn{5}{|c|}{We fix $\theta=0$}\\\hline
$\theta=0$ &  $\xi=100$&   $\xi=200$ &  $\xi=300$&$\xi=400$\\\hline
$\omega$-Jacobi& 0.9886  & 0.9886   &  0.9886   & 0.9886 \\%\hline  %   \\ \hline
Gauss-Seidel& 0.9675 &0.9675&0.9675  & 0.9675  \\
$\alpha$-CNN& 0.7660  & 0.8060 & 0.8588  &  0.7883    \\
% \end{tabular}}
% \scalebox{0.95}{\begin{tabular}[t]{|c|c|c|c|c|}
% \hline\hline
% $\xi=100$  & $\theta=0$&$\theta=\frac{\pi}{12}$&$\theta=\frac{\pi}{6}$&$\theta=\frac{\pi}{4}$\\\hline
% $\omega$-Jacobi & 0.9886   & 0.9913 & 0.9934 & 0.9942 \\ \hline
% % Gauss-Seidel&0.9002  &0.8634 &0.8682   & 0.9564 & 0.8566 &0.7776&0.7643 \\\hline
% $\alpha$-CNN& 0.7660 & 0.7743 & 0.9651 & 0.9728\\\hline
\end{tabular}

\vspace{1em}

\begin{tabular}[t]{c|c|c|c|c}
% \hline
% % \multicolumn{5}{|c|}{We fix $\theta=0$}\\\hline
% $\theta=0$ &  $\xi=100$&   $\xi=200$ &  $\xi=300$&$\xi=400$\\\hline
% $\omega$-Jacobi& 0.9886  & 0.9886   &  0.9886   & 0.9886 \\\hline  %   \\ \hline
% % Gauss-Seidel&0.9564  &0.9755 &0.9820   & 0.9853 & 0.9873  \\\hline
% $\alpha$-CNN& 0.7660  & 0.8060 & 0.8588  &  0.7883    \\
% % \end{tabular}}
% % \scalebox{0.95}{\begin{tabular}[t]{|c|c|c|c|c|}
% \hline\hline
$\xi=100$  & $\theta=0$&$\theta={\pi}/{12}$&$\theta={\pi}/{6}$&$\theta={\pi}/{4}$\\\hline
$\omega$-Jacobi & 0.9886   & 0.9913 & 0.9934 & 0.9942 \\ %\hline
Gauss-Seidel&0.9675 &0.9748 & 0.9807  & 0.9833 \\
$\alpha$-CNN& 0.7660 & 0.7743 & 0.9651 & 0.9728\\%\hline
\end{tabular}
%}
\caption{Ideal convergence bound \cref{eq:beta_star} corresponding to the same methods and problems in \cref{spectral-radius-1}.
%of two-grid methods using full coarsening and $\omega$-Jacobi and $6$-layered $\alpha$-CNN smoothers 
%for rotated Laplacian problems with different $\theta$ and $\xi$.
%The  grid size is $16\times 16$.
}
% $\alpha$-CNN smoothers have 6 convolution layers and both smoothers are applied one step for each iteration.}
\label{min-weak-approx}
\end{table}

% \begin{table}[!htbp]
% \scalebox{0.9}{\begin{tabular}[t]{|c|c|c|c|c|}
% \hline
%   & 0&$\frac{\pi}{12}$&$\frac{\pi}{6}$&$\frac{\pi}{4}$\\\hline
% $\omega$-Jacobi Smoother   &0.9853  & 0.9436 &0.8981 &0.8837\\ \hline
% % Gauss-Seidel&0.9002  &0.8634 &0.8682   & 0.9564 & 0.8566 &0.7776&0.7643 \\\hline
% $\alpha$-CNN Smoother& 0.6816 &0.4534  &0.4547 &0.4216\\\hline
% \end{tabular}}
% \caption{We compare the spectral radius of the iteration matrices corresponding to the multigrid solver under full coarsening scheme equipped with $\omega$-Jacobi smoothers and $\alpha$-CNN smoothers. The parameters of the problem is $\xi = 100 $ and $\theta = 0,\frac{\pi}{12},\frac{\pi}{6},\frac{\pi}{4}$. The size of the problem is $63\times 63$.}
% \label{spectral-radius-2}
% \end{table}

The results in \cref{spectral-radius-1} and \cref{min-weak-approx} show that for each rotated Laplacian problem, 
the convergence measure associated with $\alpha$-CNN smoothers 
are much smaller than those associated with $\omega$-Jacobi smoothers and Gauss-Seidel smoothers
which indicates a faster convergence can be achieved by multigrid solvers equipped with $\alpha$-CNN smoothers.

We use the same problem setting as the above tables. We consider the iterative solvers $x_{k}=Gx_{k-1}$ where $G$ is the 5-level multigrid solver.
 We compare the spectral radius of the iteration matrices  $G$ of 5-level multigrid solvers equipped with different smoothers and summarize the results in \cref{spectral-radius-5}. The results show that the smoothers can not only efficiently smooth the finest level errors but also have faster convergence overall as a 5-grid solver compared to $\omega$-Jacobi and Gauss-Seidel. Since $\omega$-Jacobi smoothers are more scalable than Gauss-Seidel smoothers, we will only compare neural smoothers with $\omega$-Jacobi smoothers in the remaining section.

\begin{table}[!htbp]
%\scalebox{0.9}
\centering
{
\begin{tabular}[t]{c|c|c|c|c}
% \hline
% \multicolumn{5}{|c|}{We fix $\theta=0$}\\\hline
$\theta=0$ &  $\xi=100$&   $\xi=200$ &  $\xi=300$&$\xi=400$\\\hline
$\omega$-Jacobi &0.9853  &0.9918   &0.9940    &0.9951\\%\hline  %   \\ \hline
 Gauss-Seidel&0.9564  &0.9755 &0.9820   & 0.9853  \\
$\alpha$-CNN & 0.6816 &0.8189 & 0.8805  &0.8936     %\\
% \end{tabular}}
% \scalebox{0.95}{\begin{tabular}[t]{|c|c|c|c|c|}
% \hline\hline
% $\xi=100$  & $\theta=0$&$\theta=\frac{\pi}{12}$&$\theta=\frac{\pi}{6}$&$\theta=\frac{\pi}{4}$\\ \hline
% $\omega$-Jacobi &0.9853  & 0.9436 &0.8981 &0.8837\\ %\hline
% % Gauss-Seidel&0.9002  &0.8634 &0.8682   & 0.9564 & 0.8566 &0.7776&0.7643 \\\hline
% $\alpha$-CNN& 0.6816 &0.4534  &0.4547 &0.4216\\%\hline
\end{tabular}

\vspace{1em}

\begin{tabular}[t]{c|c|c|c|c}
% \hline
% \multicolumn{5}{|c|}{We fix $\theta=0$}\\\hline
% $\theta=0$ &  $\xi=100$&   $\xi=200$ &  $\xi=300$&$\xi=400$\\\hline
% $\omega$-Jacobi &0.9853  &0.9918   &0.9940    &0.9951\\%\hline  %   \\ \hline
% % Gauss-Seidel&0.9564  &0.9755 &0.9820   & 0.9853 & 0.9873  \\\hline
% $\alpha$-CNN & 0.6816 &0.8189 & 0.8805  &0.8936     \\
% % \end{tabular}}
% % \scalebox{0.95}{\begin{tabular}[t]{|c|c|c|c|c|}
% \hline\hline
$\xi=100$  & $\theta=0$&$\theta={\pi}/{12}$&$\theta={\pi}/{6}$&$\theta={\pi}/{4}$\\ \hline
$\omega$-Jacobi &0.9853  & 0.9436 &0.8981 &0.8837\\ %\hline
Gauss-Seidel& 0.9564 & 0.8566 &0.7776&0.7643 \\
$\alpha$-CNN& 0.6816 &0.4534  &0.4547 &0.4216\\%\hline
\end{tabular}
}
\caption{Spectral radius of the iteration matrices corresponding to the 5-level multigrid methods with full coarsening and $\omega$-Jacobi with $\omega=\frac{2}{3}$, Gauss-Seidel and $6$-layered $\alpha$-CNN smoother 
for rotated Laplacian problems with different $\theta$ and $\xi$.
The  mesh size is $16\times 16$.}
\label{spectral-radius-5}
\end{table}

\paragraph{Smoothing property}

To show that our proposed method can learn the optimal smoother with the best smoothing property, for each eigenvector $v$ (that has the unit 2-norm) 
of the fine-level operator $A$ 
associated with parameters $\theta=\frac{5\pi}{12}$,  $\xi=100$, $N=16$ on a square domain, 
we compute its convergence factor $\norm{v-H^{(0)}(Av)}_2$, where $H^{(0)}$ is the smoother on the finest level. An efficient smoother should lead to small convergence factors for eigenvectors associated with larger eigenvalues.
The results are shown in \cref{all-eigs}, where the eigenmodes are listed in the descending order of the
corresponding eigenvalues.
The CNN smoother can reduce low-frequency errors more rapidly than $\omega$-Jacobi, however,
both of them have comparable performance for damping high-frequency errors.  
In contrast, $\alpha$-CNN has the best performance, which exhibits a superior smoothing property as the convergence factors of eigenvectors associated the large eigenvalues are about $6$ times smaller than those associated with the other two smoothers.

% \begin{figure}[ht]
% \vskip 0.2in
% \begin{center}
% \centerline{\includegraphics[width=\columnwidth]{all eigenvectors.png}}
% \caption{We compare the vector norms after applying Jacobi smoother, non-adaptive network smoother and adaptive smoother to all the eigenvectors of the problem matrix $A$. The eigenvectors are ordered by the corresponding eigenvalue in a decreasing order.}
% \label{all-eigs}
% \end{center}
% \vskip -0.2in
% \end{figure}

\begin{figure}[h]
% \vskip 0.2in
\begin{center}
\centerline{\includegraphics[width=0.7\columnwidth]{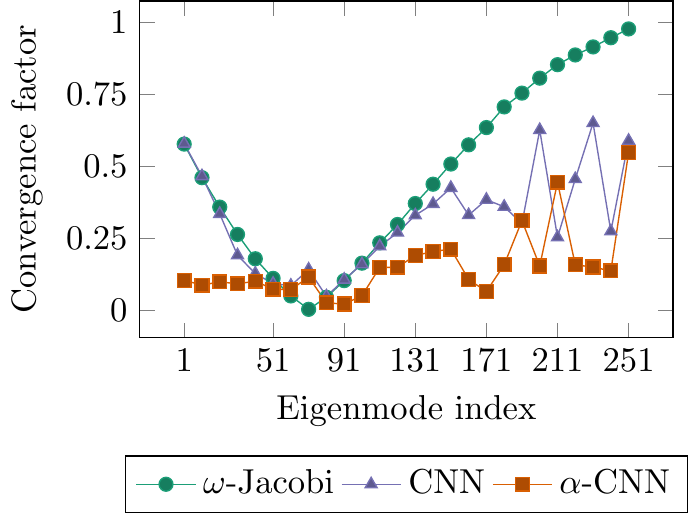}}
\caption{Convergence factors  of $\omega$-Jacobi with $\omega=\frac{2}{3}$, CNN and $\alpha$-CNN smoothers to the eigenvectors of 
$A$ for the rotated Laplacian on a $16\times 16$ grid, where $\theta=\frac{5\pi}{12}$ and $\xi=100$. 
The eigenvectors are sorted in the descending order of the corresponding eigenvalues. }
\label{all-eigs}
\end{center}
 \vskip -1.5em
\end{figure}

\paragraph{Generalization property}
To illustrate that our proposed method is useful, besides showing the statistics, we present the actual iteration numbers and runtime for multigrid solvers to converge. Also for a given PDE problem, we want to only train the neural smoothers once, that is, the neural smoothers need not to be retrained if we increase the grid size or change the geometry of the problem.
In this experiment, we first show that the trained smoothers can be generalized to different grid sizes without retraining. We fix the parameter of the problems to be  $\xi=100$ 
and $\theta=\frac{5\pi}{12}$ on one square domain. We show in \cref{exp1-fig-1} that for problems of size $1023^{2}$,  multigrid methods using $\alpha$-CNN smoothers 
converge faster in terms of the number of iterations than multigrid methods using CNN and  $\omega$-Jacobi smoothers by factors of
$1.5$ and $3.5$ respectively. 
Since the cost of applying $\alpha$-CNN 
smoothers is more than $\omega$-Jacobi, 
the time for iterations of multigrid methods using $\alpha$-CNN is only
faster than that using CNN and  $\omega$-Jacobi
by factors of $1.68$ and $2.1$, respectively.

\begin{figure}[h]
% \vskip -0.4in
\begin{center}
\includegraphics[width=0.45\columnwidth]{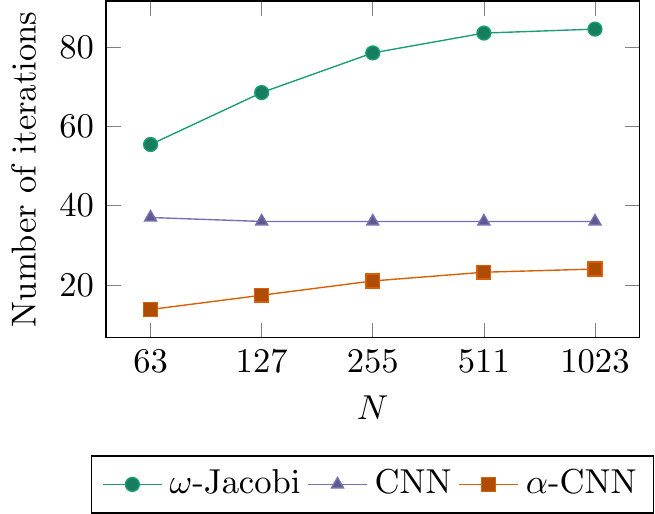} \hfill
\includegraphics[width=0.45\columnwidth]{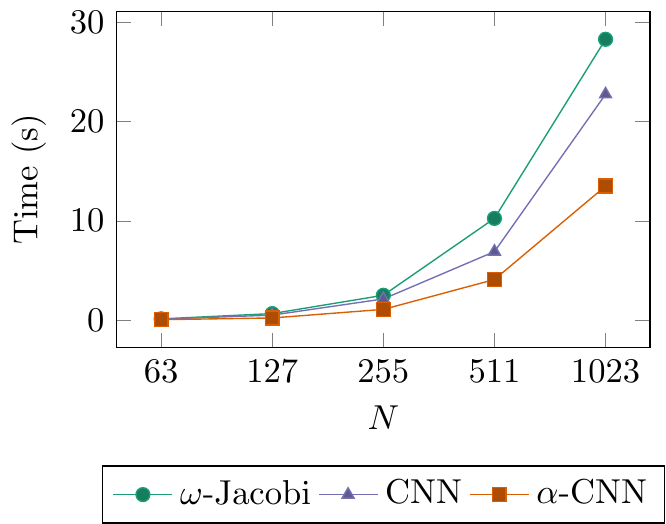}
\caption{Numbers of iterations and runtime required by multigrid with full coarsening to reach the convergence tolerance $10^{-6}$ for solving the rotated Laplacian problem of grid sizes $63^{2}$, $127^{2}$, $255^{2}$, $511^{2}$ and $1023^{2}$, with parameters $\xi=100$ and $\theta=\frac{5\pi}{12}$ on square domains.
%The multigrid methods have $5$ levels using full coarsening, equipped with $\omega$-Jacobi, CNN and $\alpha$-CNN smoothers.
%All the smoothers are applied for one iteration in both the pre- and post-smoothing steps.
}
\label{exp1-fig-1}
\end{center}
% \vskip -0.2in
\end{figure}

Since CNN smoothers were trained independently, 
they are not as successful as $\alpha$-CNN to capture the smoothing property of reducing errors that cannot be reduced by lower levels of multigrid.

Hence, 
we only compare $\alpha$-CNN and $\omega$-Jacobi smoothers in the rest of the paper. 
Next we fix the parameters of the problems to be $\theta=\frac{\pi}{4}$,  $\xi=100$ and show that the trained 
$\alpha$-CNN smoothers can be generalized to problems with two different geometries (shown in \cref{domains}) without retraining. 

\begin{figure}[h]
  \centering
  \includegraphics[width=0.8\linewidth]{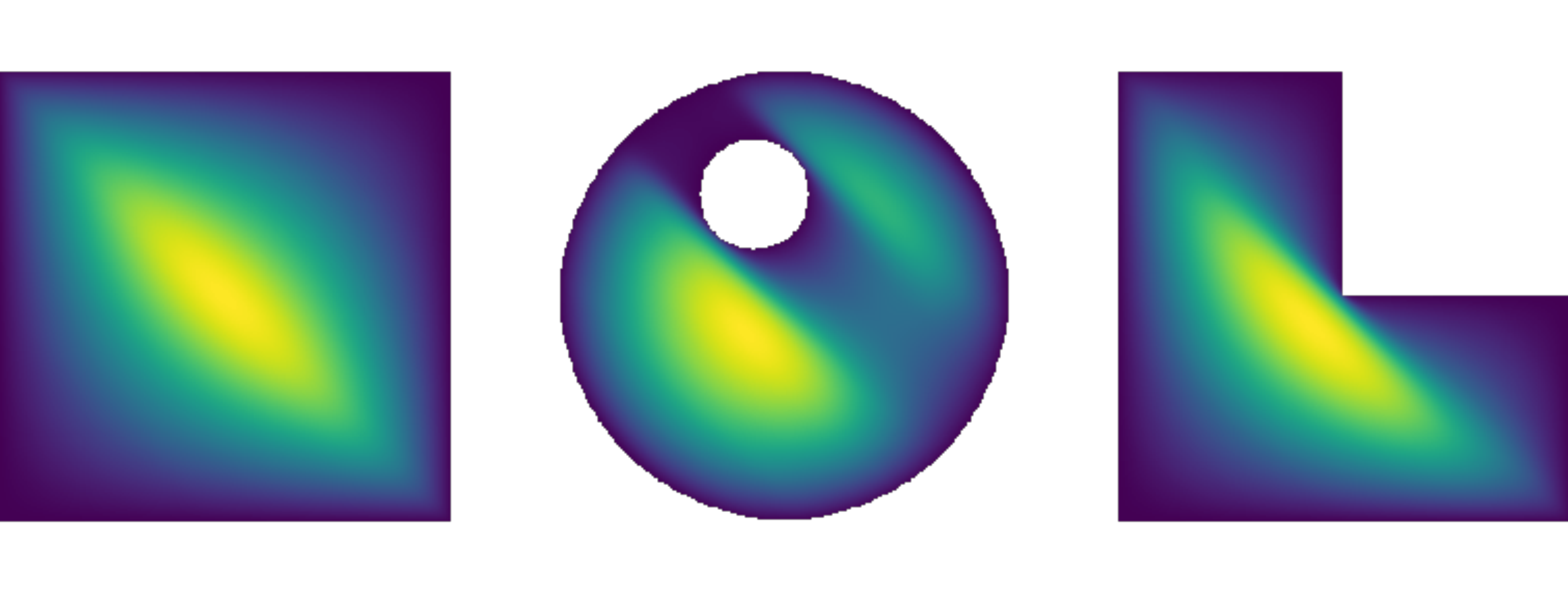}
  \caption{Ground truth solutions on square domain,  cylinder domain and L-shaped domain.}
  \label{domains}
% \vskip 0.2in
\end{figure}

\begin{figure}[h]
% \vskip -0.1in
\begin{center}
\includegraphics[width=0.45\columnwidth]{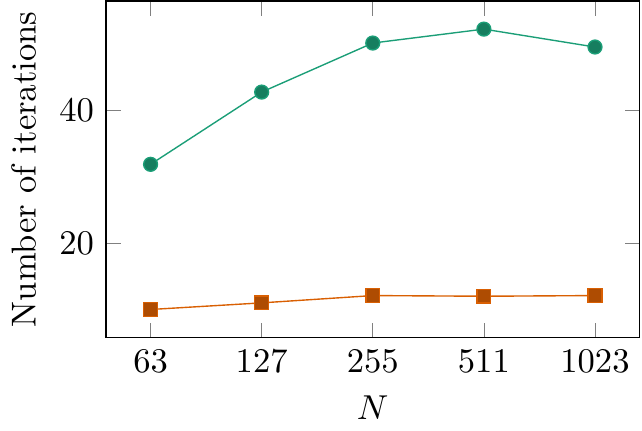} \hfill
\includegraphics[width=0.45\columnwidth]{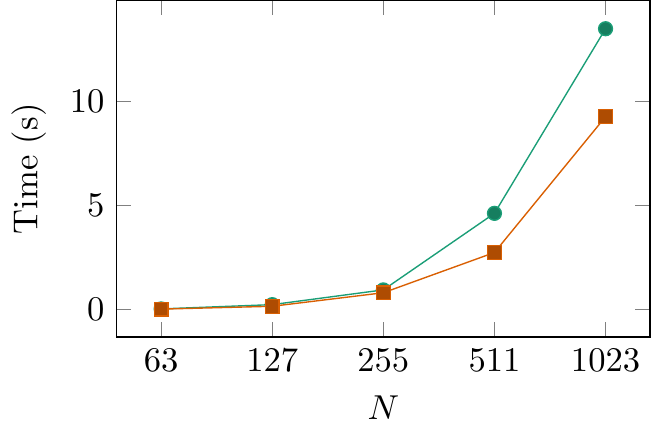} \\[0.25em]
\includegraphics[width=0.45\columnwidth]{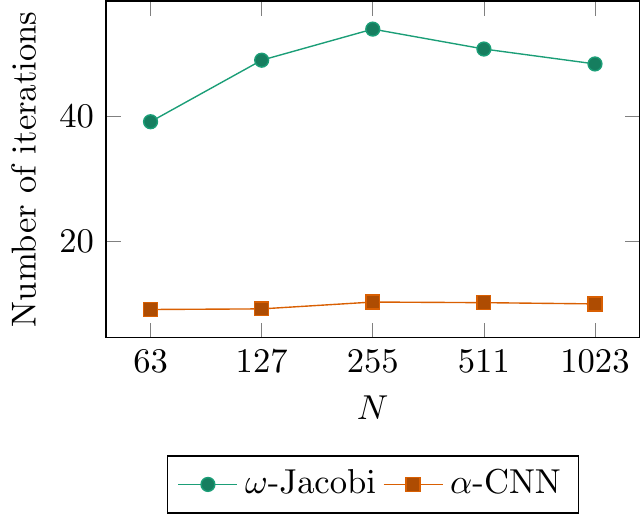} \hfill
\includegraphics[width=0.45\columnwidth]{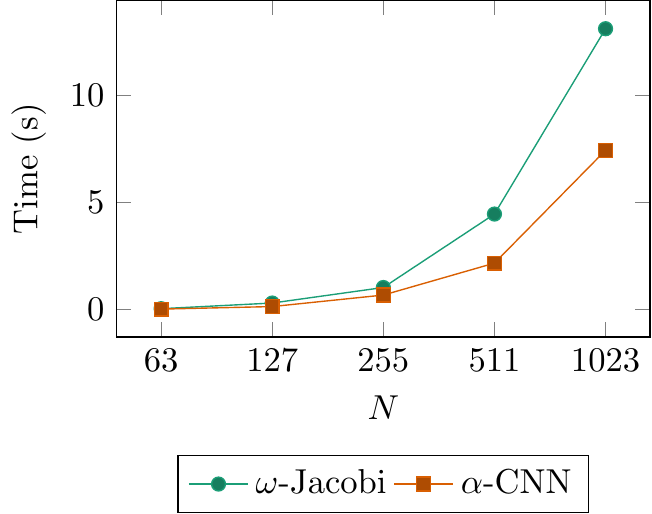}
\caption{
Numbers of iterations and runtime required by multigrid solvers for solving the rotated Laplacian problems with parameters $\theta=\frac{\pi}{4}$ and $\xi=100$ on the cylinder domain (top two figures) and the L-shaped domain (bottom two figures).}

\label{L-shape-fig1}
\end{center}
% \vskip -0.2in
\end{figure}

The results for the two different domains are shown in \cref{L-shape-fig1}. 
We can see that since we are using the convolutional layers to approximate the inverse of the smoothers,  
$\alpha$-CNN  use the information in the neighborhood information to smooth the error at each grid point 
and therefore without retraining,  
the smoother trained on square domain  can still lead  multigrid methods to converge $4.1$ times 
faster in terms of the number of iterations and $1.5$ times faster in time-to-solution on the cylinder domain for problems of size $1023^{2}$.
On the L-shaped domain for the same sized problem, the performance improvement is
$4.9$ times and $1.8$ times faster in terms of the number of iterations and the time for iterations. 

We show in \cref{red-black-1} that our proposed method can learn optimized smoothers for a variety of problems 
given by different parameters on square domain and is not restricted to the choice of coarsening schemes in multigrid. 
In particular, for $\theta=\frac{5\pi}{12}$, with full coarsening,  the
multigrid method using $\alpha$-CNN smoothers is $19.2$ times faster in terms of the number of iterations and achieves 
a speedup of factor $4.4$ in the  time for iterations.
When red-black coarsening scheme is used, multigrid solver with $\alpha$-CNN smoothers 
can still require much fewer iterations than the one with $\omega$-Jacobi by $1.9$ times,
and converges about $1.3$ times faster in time.  

\begin{figure}[h]
% \vskip -0.1in
\begin{center}
\includegraphics[width=0.45\columnwidth]{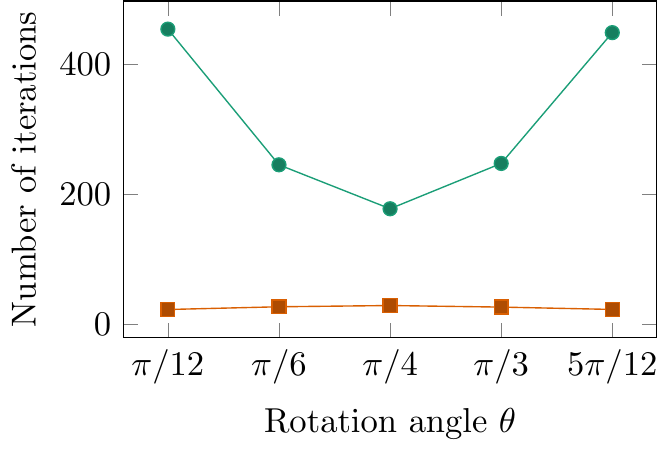} \hfill
\includegraphics[width=0.45\columnwidth]{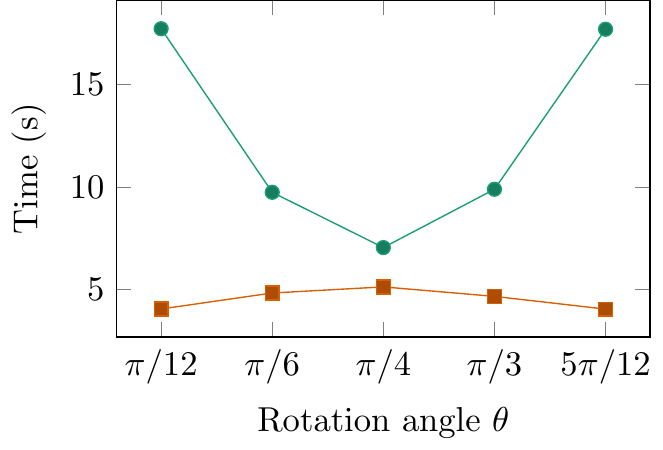} \\[0.25em]
\includegraphics[width=0.45\columnwidth]{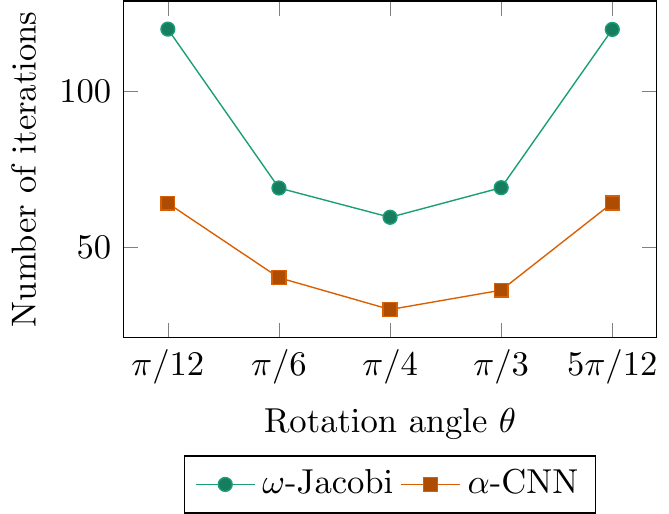} \hfill
\includegraphics[width=0.45\columnwidth]{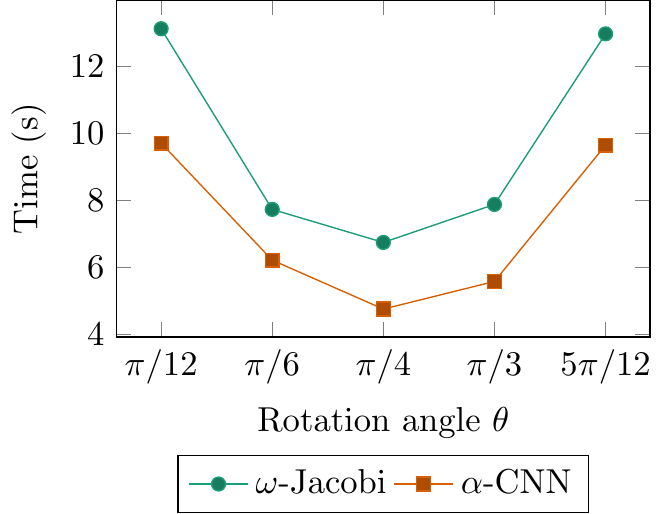}
\caption{Numbers of iterations  and runtime required by multigrid solvers for solving the rotated Laplacian problems of size $n=511^{2}$ with
 $\theta=[\frac{\pi}{12}, \frac{\pi}{6}, \frac{\pi}{4}, \frac{\pi}{3}, \frac{5\pi}{12}]$ and $\xi=100$.
%with 5 levels equipped with $\omega$-Jacobi smoothers and $\alpha$-CNN smoothers on problems of parameter $\theta=[\frac{\pi}{12},\frac{\pi}{6},\frac{\pi}{4},\frac{\pi}{3},\frac{5\pi}{12}]$  to reach the convergence tolerance 1e-6. 
The top two figures show the performance of multigrid with full coarsening and the bottom two figures show the performance of multigrid  with red-black coarsening. 
}
\label{red-black-1}
\end{center}
% \vskip -0.2in
\end{figure}

Next, we show that we can learn a single smoother which works for all the problems discussed above. Instead of training a smoother for each problem individually, we construct a training set which contains the problems for $\theta=[\frac{\pi}{12}, \frac{\pi}{6}, \frac{\pi}{4}, \frac{\pi}{3}, \frac{5\pi}{12}]$ and $\xi=100$. We show in \cref{mixed} that the performance of using a single smoother for all the problems is slightly worse than training smoothers individually but still outperforms $\omega$-Jacobi.

\begin{figure}[h]
% \vskip -0.1in
\begin{center}
\includegraphics[width=0.45\columnwidth]{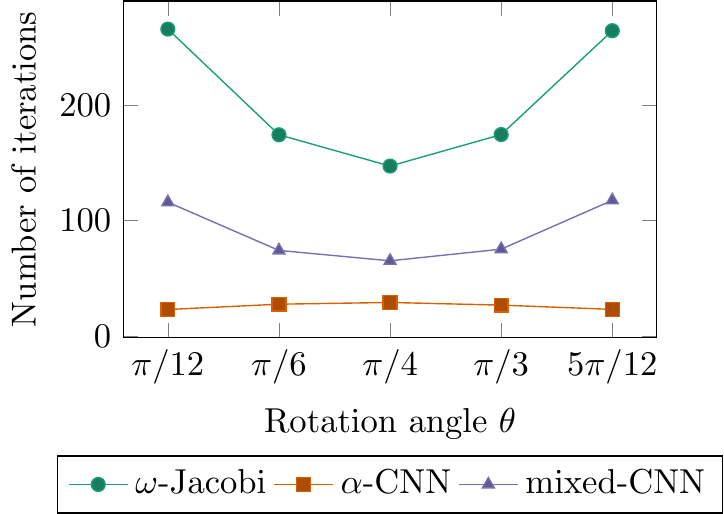} 

\caption{Numbers of iterations  and runtime required by multigrid solvers with full coarsening for solving the rotated Laplacian problems of size $n=511^{2}$ with
 $\theta=[\frac{\pi}{12}, \frac{\pi}{6}, \frac{\pi}{4}, \frac{\pi}{3}, \frac{5\pi}{12}]$ and $\xi=100$ when the smoother is trained from a dataset containing  problems with different  $\theta$ and $\xi$.  
%with 5 levels equipped with $\omega$-Jacobi smoothers and $\alpha$-CNN smoothers on problems of parameter $\theta=[\frac{\pi}{12},\frac{\pi}{6},\frac{\pi}{4},\frac{\pi}{3},\frac{5\pi}{12}]$  to reach the convergence tolerance 1e-6. 
}
\label{mixed}
\end{center}
% \vskip -0.2in
\end{figure}

Finally, \cref{red-black-2} shows the performance of  a 5-level multigrid with $\omega$-Jacobi smoothers and $\alpha$-CNN smoothers using full coarsening and red-black coarsening with the same problem setting as in \cref{red-black-1}. However, since we are using 6 convolutional layers with full coarsening and 2 convolutional layers with red-black coarsening. For fair comparison in terms of computational cost per iteration, in this experiment we run 6 Jacobi steps each iteration for full coarsening and 2 Jacobi steps for red-black coarsening.  

\begin{figure}[h]
% \vskip -0.1in
\begin{center}
\includegraphics[width=0.45\columnwidth]{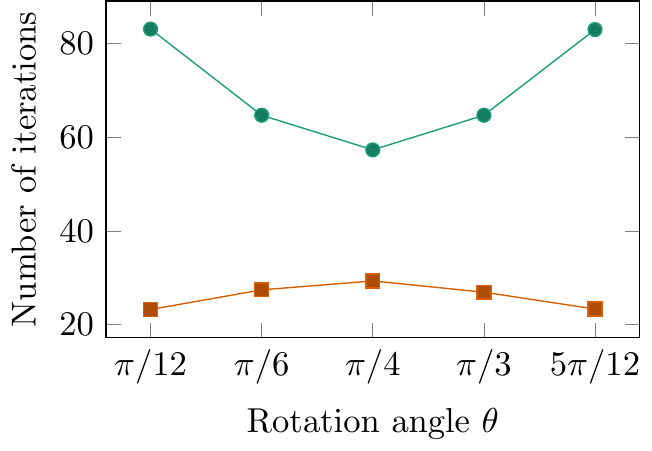} \hfill
\includegraphics[width=0.45\columnwidth]{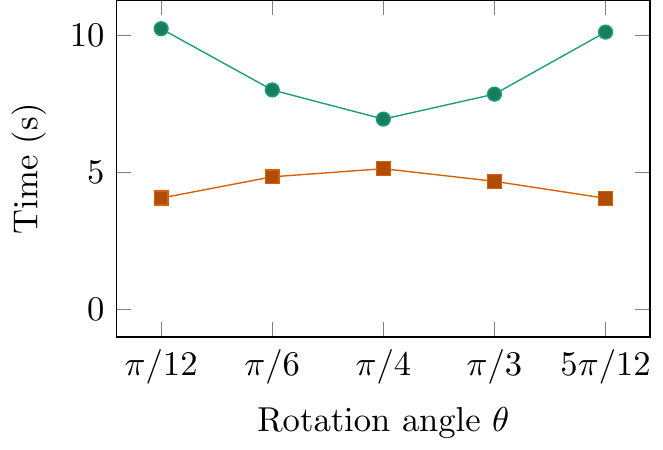} \\[0.25em]
\includegraphics[width=0.45\columnwidth]{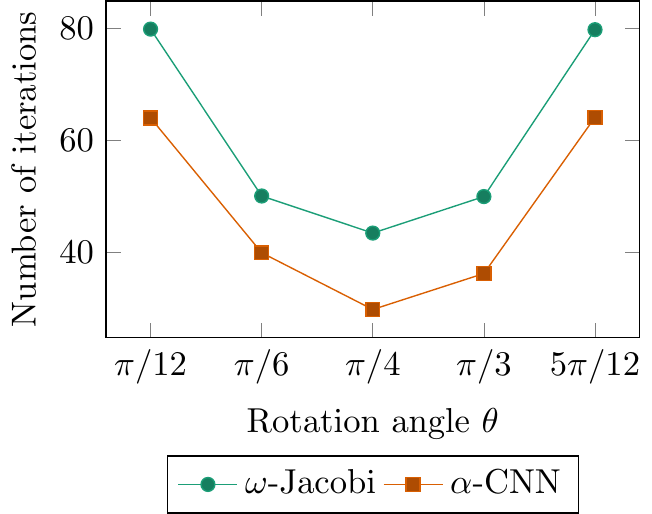} \hfill
\includegraphics[width=0.45\columnwidth]{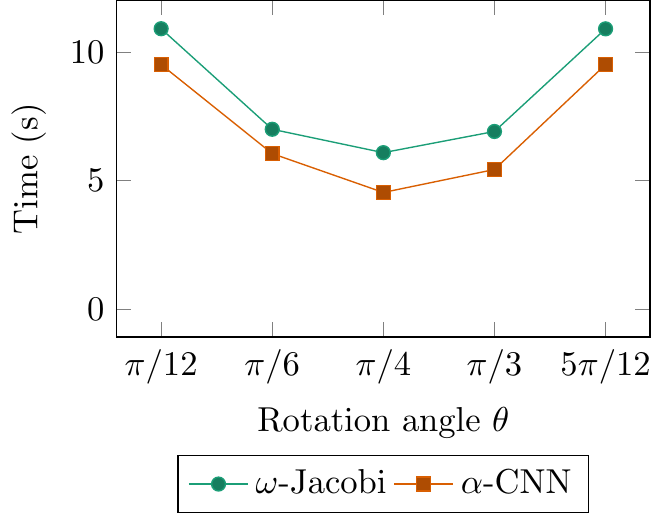}
\caption{Numbers of iterations  and runtime required by multigrid solvers for solving the rotated Laplacian problems of size $n=511^{2}$ with
 $\theta=[\frac{\pi}{12}, \frac{\pi}{6}, \frac{\pi}{4}, \frac{\pi}{3}, \frac{5\pi}{12}]$ and $\xi=100$.
%with 5 levels equipped with $\omega$-Jacobi smoothers and $\alpha$-CNN smoothers on problems of parameter $\theta=[\frac{\pi}{12},\frac{\pi}{6},\frac{\pi}{4},\frac{\pi}{3},\frac{5\pi}{12}]$  to reach the convergence tolerance 1e-6. 
The top two figures show the performance of multigrid with full coarsening and the bottom two figures show the performance of multigrid  with red-black coarsening. 
}
\label{red-black-2}
\end{center}
% \vskip -0.2in
\end{figure}

\subsection{Variable coefficient PDEs}
We then consider the following variable coefficient problem:

$$-\nabla \cdot ((\sin{\kappa\pi xy}+1.1)\nabla u(x,y)) = f(x,y),$$
which is determined by the frequency $\kappa$.

% \begin{table}
% \centering
% \begin{tabular}{c|c|c|c}
% % \hline
% % \multicolumn{5}{|c|}{We fix $\theta=0$}\\\hline

% \hline
%  & 1 & 2 & 3\\ \hline
% Convolutional  &    &  &\\
% Fully connected&  80/21  &   &11/21\\\hline%\hline
% \end{tabular}
% \caption{problem 1}
% \label{spectral-radius-3}
% \end{table}

In this experiment we consider solving the problems determined by $\kappa=0.1,1,10,$ and $100$. For each problem we consider a 4-level multigrid solver. 
We use the two approaches discussed in Section \ref{sec:VariableCoeff} to learn one single convolution kernel of size $3\times 3$ used for smoothing. We use 4 fully connected layers with 40 neurons for each layer for the first approach which has $6,800$ parameters to train in total.  We use 3 convolutional layers which has 7, 5 and 3 channels for each layer and a fully connected layer of size $27\times 9$ which has 378 parameters to train in total. We use Leaky ReLu activation function to perform a nonlinear mapping of the stencils to the smoother. We train the smoothers on problems of size $31\times 31$ and test the performance on problems of size $255\times 255$. 

We compare the performance of using different approaches for learning $\alpha$-CNN smoothers with weighted Jacobi and show the results in \cref{tbl:var-iter} and \cref{tbl:var-time}. The fully connected approach has similar performance in terms of both iteration number and runtime compared to the convolutional approach while having 17 times  more parameters. Both $\alpha$-CNN approaches can achieve $2\times$ speedup in terms of iteration number and $1.6\times$ speedup in terms of runtime.

\begin{table}
\centering
\begin{tabular}{c|c|c|c|c}
% \hline
% \multicolumn{5}{|c|}{We fix $\theta=0$}\\\hline
%\hline
 & $k=0.1$ & $k=1$ & $k=10$ & $k=100$ \\ \hline
Convolutional  & 6 & 7&11 &30  \\
Fully connected& 6   &6  &10  & 28 \\\hline%\hline
$\omega$-Jacobi&  17  & 17 & 20 & 63 %\\\hline%\hline
\end{tabular}
\caption{Numbers of iterations required by multigrid solvers for solving the variable coefficient problems of size $n=255^{2}$ with
 $\kappa=0.1,1,10,100$ using $\alpha$-CNN and $\omega$-Jacobi with $\omega=\frac{2}{3}$.}
\label{tbl:var-iter}
\end{table}

\begin{table}
\centering
\begin{tabular}{c|c|c|c|c}
% \hline
% \multicolumn{5}{|c|}{We fix $\theta=0$}\\\hline
%\hline
 & $k=0.1$ & $k=1$ & $k=10$ & $k=100$ \\ \hline
Convolutional  &  0.1016    & 0.1135  & 0.1514 & 0.3329 \\
Fully connected& 0.1027  & 0.1014  & 0.1417 & 0.3147 \\\hline%\hline
$\omega$-Jacobi&  0.1693  &  0.1671  & 0.1913  & 0.4997 %\\\hline%\hline
\end{tabular}
\caption{Run time required by multigrid solvers for solving the variable coefficient problems of size $n=255^{2}$ with
 $\kappa=0.1,1,10,100$ using $\alpha$-CNN and $\omega$-Jacobi with $\omega=\frac{2}{3}$.}
\label{tbl:var-time}
\end{table}

% \begin{figure}[h]

% \vspace{1cm}
% \centering\includegraphics[width=0.7\columnwidth]{variable_coefficient_all_eig.png}
% \caption{Convergence factors  of $\omega$-Jacobi, Gauss Seidel and $\alpha$-CNN smoothers to the eigenvectors of 
% $A$ for the variable coefficient problem. 
% The eigenvectors are sorted in the descending order of the corresponding eigenvalues. }
% \label{fig:variable}
% \end{figure}

\subsection{Incorporation with FGMRES}
In this section we use multigrid solvers  as preconditioners of flexible GMRES on the same group of problems as in \cref{red-black-1}.  Notice that due to the use of nonlinear activation functions in the neural smoothers,  it is mandatory to use flexible GMRES instead of standard GMRES as the accelerator.  We compare the performance of using the $\alpha$-CNN smoothers trained before and using the $\omega$-Jacobi smoothers in terms of iteration numbers and running time. We show the results in \cref{tbl:fgmres} that using 
$\alpha$-CNN can achieve up to 3.36$\times$ improvement in terms of iteration number and up to 1.5$\times$ improvement in terms of time compare to $\omega$-Jacobi. 

\begin{table}[!htbp]
\centering
\begin{tabular}[t]{c|c|c|c|c|c}
% \hline
% \multicolumn{5}{|c|}{We fix $\theta=0$}\\\hline
% $\theta=0$ &  $\xi=100$&   $\xi=200$ &  $\xi=300$&$\xi=400$\\\hline
% $\omega$-Jacobi &0.9853  &0.9918   &0.9940    &0.9951\\%\hline  %   \\ \hline
% % Gauss-Seidel&0.9564  &0.9755 &0.9820   & 0.9853 & 0.9873  \\\hline
% $\alpha$-CNN & 0.6816 &0.8189 & 0.8805  &0.8936     \\
% % \end{tabular}}
% % \scalebox{0.95}{\begin{tabular}[t]{|c|c|c|c|c|}
% \hline\hline
$\xi=100$  &$\theta={\pi}/{12}$&$\theta={\pi}/{6}$&$\theta={\pi}/{4}$& $\theta={\pi}/{3}$&$\theta={5\pi}/{12}$\\ \hline
FGMRES with $\omega$-Jacobi& 37.0  & 30.2 & 28.0  &30.0 & 37.0 \\ %\hline
FGMRES with $\alpha$-CNN&  11.0  & 12.0 & 13.0  & 12.0 &11.0 \\%\hline
\end{tabular}
\caption{Numbers of iterations required by preconditioned FGMRES to reach the convergence tolerance $10^{-6}$ for solving the rotated Laplacian problems with different $\theta$ and $\xi$. The grid size is $511^{2}$.}

\centering
\begin{tabular}[t]{c|c|c|c|c|c}
% \hline
% \multicolumn{5}{|c|}{We fix $\theta=0$}\\\hline
% $\theta=0$ &  $\xi=100$&   $\xi=200$ &  $\xi=300$&$\xi=400$\\\hline
% $\omega$-Jacobi &0.9853  &0.9918   &0.9940    &0.9951\\%\hline  %   \\ \hline
% % Gauss-Seidel&0.9564  &0.9755 &0.9820   & 0.9853 & 0.9873  \\\hline
% $\alpha$-CNN & 0.6816 &0.8189 & 0.8805  &0.8936     \\
% % \end{tabular}}
% % \scalebox{0.95}{\begin{tabular}[t]{|c|c|c|c|c|}
% \hline\hline
$\xi=100$  &$\theta={\pi}/{12}$&$\theta={\pi}/{6}$&$\theta={\pi}/{4}$& $\theta={\pi}/{3}$&$\theta={5\pi}/{12}$\\ \hline
FGMRES with $\omega$-Jacobi& 3.48  & 2.86 & 2.74  &2.65 & 3.46 \\ %\hline
FGMRES with $\alpha$-CNN& 2.32   & 2.52 & 2.56    &2.47  & 2.29 \\%\hline
\end{tabular}
\caption{Run time required by preconditioned FGMRES to reach the convergence tolerance $10^{-6}$ for solving the rotated Laplacian problems with different $\theta$ and $\xi$. The grid size is $511^{2}$.}
\label{tbl:fgmres}
\end{table}

\section{Conclusion}
\label{sec:conclusion}
In this work we propose an efficient framework for training smoothers 
in the form of multi-layered CNNs that can be equipped by multigrid methods 
for solving linear systems arising from PDE problems.
The training process of the proposed smoothing algorithm, called $\alpha$-CNN, is guided by multigrid convergence theories and have
the desired property of minimizing  errors that cannot be efficiently annihilated by coarse-grid corrections.
Experiments on rotated Laplacian problems show superior smoothing property of $\alpha$-CNN smoothers
that leads to better performance of multigrid convergence when
combined with standard coarsening and interpolation schemes compared with
classical relaxation-based smoothers. 
We also show that well-trained $\alpha$-CNN smoothers  
on small problems can be generalized to problems of much larger sizes and 
different geometries without retraining. 
For future work,
we plan to use graph convolution networks to extend the current framework to unstructured meshes and study how to optimize other components  in multigrid solvers such as 
coarsening algorithms and grid transfer operators.

\section*{Acknowledgments}
We would like to acknowledge the fruitful discussions with the \textit{hypre} team at LLNL, which prompted the exploration of 
interpretability of the learned smoothers in
Section~\ref{inter}.
%We would also like to thank the referees whose various comments have helped to improve the paper.

\bibliographystyle{siamplain}
\bibliography{deep_solver}

\end{document}